\documentclass[leqno,11pt,a4paper]{amsart}

\usepackage{jc}

\newcommand{\limitsto}{\dashrightarrow}

\newcommand{\NW}{\on{NW}}

\begin{document}

\author{Julian Chaidez}
\address{Department of Mathematics\\University of Southern California\\Los Angeles, CA\\90007\\USA}
\email{julian.chaidez@usc.edu}

\title[Conformally Symplectic Topology From A Dynamical Viewpoint]{Conformally Symplectic Topology From A Dynamical Viewpoint}

\begin{abstract} This survey article discusses the emerging interaction between conformally symplectic topology and dynamics, with a focus on recent developments in convex hypersurface theory. 
\end{abstract}

\vspace*{-10pt}

\maketitle

%\tableofcontents

\vspace*{-20pt}

\section{Introduction} \label{sec:introduction}

Since the emergence of Floer theory in the 1980s, the interplay between symplectic topology and dynamics has been a fundamental theme of symplectic geometry. The first achievement in this spirit was Floer's celebrated proof of the Arnold Conjecture for aspherical symplectic manifolds \cite{floer1988morse,floer1988unregularized,floer1995transversality}. In subsequent decades, many remarkable results in symplectic dynamics have been achieved by applying Floer-theoretic methods. Examples include the proof of the Conley conjecture \cite{ginzburg2010conley} for aspherical symplectic manifolds; the proof of the Hofer-Zehnder conjecture \cite{shelukhin2022hofer} for monotone symplectic manifolds with semi-simple quantum cohomology; the proof of the Weinstein conjecture in dimension three \cite{taubes2007seiberg,hutchings2009weinstein} and in other settings \cite{floer1990weinstein,albers2009weinstein}; and smooth closing lemmas in symplectic dynamics \cite{i2015,ai2016,eh2021,cpz2021,chaidez2022contact}. This list, by no means comprehensive, highlights the remarkable dynamical applications of Floer theoretic methods.

\vspace{3pt}

Classical symplectic and Hamiltonian dynamics deals with maps and flows that preserve a symplectic form (e.g. Hamiltonian diffeomorphisms and flows) or that conserve a contact form (e.g. Reeb flows). However, there is a natural, broader class of dynamical systems in symplectic topology that has (until recently) attracted much less attention from the symplectic community. 

\begin{definition*} A diffeomorphism $\Phi$ on a manifold $M$ is \emph{conformally symplectic} if it preserves a symplectic (or Liouville, or contact) form $\Omega$ on $M$ up to a scaling factor.
\[
\Phi^*\Omega = F \cdot \Omega \qquad\text{for a smooth function }F:M \to \R_+
\]
\end{definition*}
\noindent Conformally symplectic dynamics is not nearly as well studied as its strictly symplectic counterpart. Relatively little is is known about the properties of these systems, e.g. their fixed points, invariant sets, ergodic theory and so on. However, conformally symplectic dynamics is quite ubiquitous, appearing naturally in many parts of symplectic topology. Prototypical examples include the Liouville flow on an exact symplectic manifold, contactomorphisms and the flow of the characteristic foliation on hypersurfaces in contact manifolds. Moreover, several interesting questions of a purely topological nature about Liouville manifolds, contact manifolds and convex hypersurfaces are closely related to dynamical questions in conformally symplectic dynamics. 

\vspace{3pt}

The interaction between conformally symplectic topology and dynamics therefore represents exciting and largely uncharted mathematical territory. This short article serves as an invitation to the subject, with a focus on the specific setting of characteristic foliations. We start by introducing contact Hamiltonian manifolds and their characteristic foliations (Section \ref{sec:contact_Hamiltonian_manifolds}). These are intrinsic versions of hypersurfaces in contact manifolds, and they give rise to many of the most natural examples of conformally symplectic dynamical systems that the author is acquainted with. We then discuss convexity in the sense of Giroux \cite{g1991} in this intrinsic language, and give an overview of recent developments in the problem of convex approximation from the dynamical perspective (Section \ref{sec:convexity}). We conclude with a discussion of several conjectures and open problems (Section \ref{sec:questions_and_conjectures}).

\subsection*{Acknowledgements} This article was written for the Proceedings of the 2026 Georgia International Topology Conference. The author was partially supported by National Science Foundation award DMS-2446019 and by the US-Israel Binational Science Foundation award 2024157. This article is dedicated to the memory of Dietmar Salamon, one of the most important mathematical influences on the author.

\newpage

\section{Contact Hamiltonian Manifolds And Characteristic Flows} \label{sec:contact_Hamiltonian_manifolds}

Contact Hamiltonian manifolds are the natural intrinsic models for hypersurfaces in contact manifolds. They are directly analogous to Hamiltonian manifolds, which model hypersurfaces in symplectic manifolds (cf. Fish-Hofer \cite{fish2023feral}). Moreover, they possess a natural dynamical system called the characteristic foliation, which will be the main object of interest in this survey. In this section, we introduce contact Hamiltonian manifolds and their various associated structures. 

\subsection{Contact Hamiltonian Structures} We start by giving a precise definition and some examples. 

\begin{definition}[Contact Hamiltonian Form] \label{def:contact_Hamiltonian_form} A \emph{contact Hamiltonian form} $\lambda$ on a $2n$-manifold $\Sigma$ is a smooth 1-form such that
\begin{equation} \label{eq:contact_Hamiltonian_form}
d\lambda^n \neq 0 \text{ at any point where $\lambda \wedge d\lambda^{n-1}$ vanishes.}
\end{equation}
\end{definition}

\begin{lemma}[Rescaling Invariance] Let $\lambda$ be a contact Hamiltonian form on $\Sigma$ and let $f$ be a smooth nowhere vanishing function on $\Sigma$. Then $f\lambda$ is also contact Hamiltonian.
\end{lemma}

\begin{proof} This is immediate from Definition \ref{def:contact_Hamiltonian_form} and the following identities.
\[f\lambda \wedge d(f\lambda)^{n-1} = f^n \lambda \wedge d\lambda^{n-1}
 \qquad \text{and}\qquad d(f\lambda)^n|_{\on{ker}(\lambda)} = f^n \cdot d\lambda^n|_{\on{ker}(\lambda)} \qedhere\]
\end{proof}

\begin{definition}[Singular Hyperplane Field] A \emph{singular hyperplane field} $\eta$ on a manifold $\Sigma$ is an equivalence class of 1-forms up to multiplication by a nowhere zero smooth function. We let
\[
\on{ker}(\lambda) \qquad \text{denote the equivalence class of a 1-form $\lambda$}
\]
A \emph{singularity} of a singular hyperplane field $\eta$ is a point where the defining 1-form $\lambda$ vanishes. \end{definition}
 
\begin{definition}[Contact Hamiltonian Manifold] \label{def:contact_Hamiltonian_manifold} A \emph{contact Hamiltonian manifold} $(\Sigma,\eta)$ is a $2n$-manifold $\Sigma$ equipped with a singular hyperplane field $\eta$ defined by a contact Hamiltonian form.\end{definition}

\begin{remark} This definition is more general than the version given in the authors work \cite{jc2024}, where contact Hamiltonian manifolds are synonymous with even contact manifolds. 
\end{remark}

\begin{figure}[h]
    \centering
    \includegraphics[width=.9\linewidth]{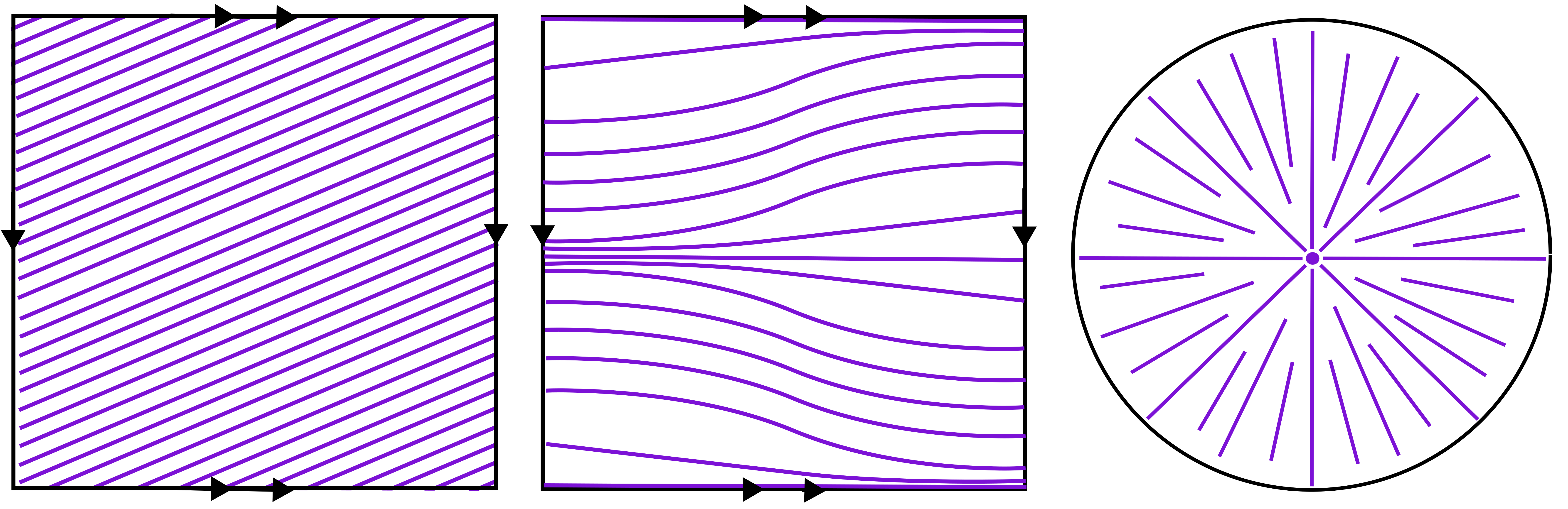}
    \caption{Examples of contact Hamiltonian structures on the 2-torus and the disk. }
    \label{fig:contact_Hamiltonian_examples}
\end{figure}

\noindent Here are the motivating examples of contact Hamiltonian manifolds that arise in practice.

\begin{example}[Hypersurfaces] \label{ex:hypersurfaces} Fix a hypersurface $\Sigma$ in a contact manifold $(Y,\xi)$ with contact form $\alpha$. Then $\Sigma$ is a contact Hamiltonian manifold with contact Hamiltonian structure
\[
\eta = \xi \cap T\Sigma \qquad\text{defined by the contact Hamiltonian form}\qquad \lambda = \alpha|_\Sigma
\]
Indeed, note that $\lambda \wedge d\lambda^{n-1}$ vanishes at a point in $\Sigma$ if and only if $\lambda$ vanishes at that point. In this case $\xi = T\Sigma$ at that point and thus $d\lambda|_{T\Sigma}$ is symplectic. \end{example}

\begin{example}[Exact Symplectic Manifolds] \label{ex:exact_symplectic} Any exact symplectic manifold $(\Sigma,\lambda)$ is contact Hamiltonian with contact Hamiltonian form $\lambda$ since $d\lambda^n$ is nowhere vanishing. \end{example}

\begin{example}[Even Contact Manifolds] \label{ex:even_contact} Recall that an even contact manifold $(\Sigma,\eta)$ is a $2n$-manifold with a maximally non-degenerate (non-singular) hyperplane field $\eta$. Equivalently
\[
\eta = \on{ker}(\lambda) \qquad\text{where}\qquad \lambda \wedge d\lambda^{n-1} \text{ is nowhere vanishing}
\]
Thus an even contact manifold is precisely a contact Hamiltonian manifold where $\eta$ is nowhere singular.\end{example}

\begin{example}[Symplectization] \label{ex:symplectization} The symplectization $\R_s \times \Gamma$ of a contact manifold $(\Gamma,\xi)$ is naturally a contact Hamiltonian manifold with contact Hamiltonian structure
\[
\eta = \on{span}(\partial_s) \oplus \xi
\]
This is a special case of both Example \ref{ex:exact_symplectic} and Example \ref{ex:even_contact}. \end{example}

\begin{example}[Suspensions] \label{ex:suspensions} Let $\Phi$ be a contactomorphism of a contact manifold $(\Gamma,\xi)$. Recall that the suspension of $\Phi$ is the quotient
\[\Sigma(\Phi) = (\R_s \times \Gamma)/\sim\]
by the identification sending $(s,x)$ to the point $(s-1,\Phi(x))$. This preserves the contact Hamiltonian structure on $\R_s \times \Gamma$, which then descends to the suspension. \end{example}

\begin{exercise} Find a contact Hamiltonian 1-form defining each of the 1-dimensional singular hyperplane fields (or in this case, line fields) depicted in Figure \ref{fig:contact_Hamiltonian_examples}.
\end{exercise}

We are often interested in understanding contact Hamiltonian structures up to deformation (or equivalently, isotopy). There are a two different notions of isotopy that are important.

\begin{definition}[Isotopies] An isotopy $\eta_s$ of contact Hamiltonian structures on a manifold $\Sigma$ is simply a smooth family of contact Hamiltonian structures
\[
\eta_s = \on{ker}(\lambda_s) \qquad\text{for $s \in [0,1]$}
\]
where $\lambda_s$ is a smooth family of contact Hamiltonian forms. An isotopy is \emph{ambient} in a contact manifold $(Y,\xi)$ if there is a family of codimension one embeddings such that
\[
\iota_s:\Sigma \to Y \qquad\text{such that}\qquad \eta_s = \iota_s^*\xi
\]\end{definition}

\begin{remark} With respect to isotopy, contact Hamiltonian structures behave more like contact forms and Liouville forms than symplectic forms and contact structures. In particular, there is no Moser stability or Grey stability that guarantees that isotopic contact Hamiltonian structures are strictly isomorphic.  \end{remark}

\subsection{Characteristic Foliation} A contact Hamiltonian manifold has a naturally associated singular line field, which generates a naturally associated smooth flow. 

\begin{definition}[Characteristic Flow] The \emph{characteristic vector field} $Z$ determined by a contact Hamiltonian form $\lambda$ and a volume form $\mu$ is the unique vector field satisfying
\begin{equation} \label{eq:characteristic_vectorfield}
\iota_Z\mu = \lambda \wedge d\lambda^{n-1}
\end{equation}
A flow acquired by integrating a characteristic vector field is called a \emph{characteristic flow}. \end{definition}

The characteristic vector field depends on the choice of contact Hamiltonian form and volume form, but it determines a well-defined (singular) line field, in the following sense.

\begin{definition}[Singular Line Field] A \emph{singular line field} $L$ on a manifold $\Sigma$ is an equivalence class of vector fields up to multiplication by a nowhere zero smooth function. We let
\[
\on{span}(V) \qquad \text{denote the equivalence class of a vector field $V$}
\]
A \emph{singularity} of a singuar line field $L$ is a point where the defining vector field vanishes.\end{definition}

\begin{definition}[Characteristic Foliation] The \emph{characteristic foliation} $\Sigma_\eta$ of a contact Hamiltonian manifold $(\Sigma,\eta)$ is the unique singular line field spanned by any characteristic vector field. \end{definition}

\noindent The well-definedness of the characteristic foliation is given by the following elementary lemma.

\begin{lemma}[Well-Defined] The singular line field spanned by a characteristic vector field $Z$ of a contact Hamiltonian manifold $(\Sigma,\eta)$ is independent of the choice of contact Hamiltonian form and volume form.
\end{lemma}

\begin{proof} Let $Z$ be the characteristic vector field for a contact Hamiltonian form $\lambda$ and volume form $\mu$, and let $\lambda' = f\lambda$ and $\mu' = g\mu$ for nowhere vanishing functions $f$ and $g$. Then
\[
\iota_{FZ}\mu' = \lambda' \wedge (d\lambda')^{n-1} \qquad\text{where}\qquad F = f^n/g 
\]
Thus $FZ$ is the characteristic vector field of $\lambda'$ and $\mu'$. This proves the lemma.
\end{proof}

A similar argument shows that any spanning vector field of the characteristic foliation is characteristic for some contact Hamiltonian form.

\begin{lemma}[Spanning Vector Fields] \label{lem:spanning_vector_fields} Fix a contact Hamiltonian manifold $(\Sigma,\eta)$ and a volume form $\mu$. Then any vector field $Z$ spanning $\Sigma_\eta$ satisfies (\ref{eq:characteristic_vectorfield}) for some contact Hamiltonian form $\lambda$.
\end{lemma}

The characteristic foliation can be described as the kernel of a certain maximally non-degenerate bilinear form at any of the non-singular points.

\begin{lemma}[Kernel Description] \label{lem:ker_description} The singularities of the characteristic foliation $\Sigma_\eta$ of a contact Hamiltonian manifold $(\Sigma,\eta)$ agree with the singularities of $\eta$. Away from the singularities
\begin{equation} \label{eq:ker_equals_char_foliation}
\Sigma_\eta = \on{ker}(d\lambda|_\eta) \subset \eta\qquad\text{for any contact Hamiltonian form $\lambda$}
\end{equation}
\end{lemma}

\begin{proof} Fix a characteristic vector field $Z$ defined by $\lambda$ and a volume form $\mu$. For the first claim, note that $Z$ vanishes if and only if $\lambda \wedge d\lambda^{n-1}$ vanishes by (\ref{eq:characteristic_vectorfield}), and that $\lambda \wedge d\lambda^{n-1}$ vanishes if and only if $\lambda$ vanishes by (\ref{eq:contact_Hamiltonian_form}). For the second claim, note that the restriction of $d\lambda$ to $\eta = \on{ker}(\lambda)$ must have a rank one kernel at any point where $\lambda \wedge d\lambda^{n-1}$ is non-zero. Moreover, any non-zero vector $W$ tangent to this kernel satisfies
\begin{equation} \label{eq:char_vectorfield_f_version}
\iota_W(\lambda \wedge d\lambda^{n-1}) = \iota_W\lambda \cdot d\lambda^{n-1} + \lambda \wedge \iota_W d\lambda \wedge d\lambda^{n-2} = 0
\end{equation}
It follows that $Z$ and $W$ are proportional at any non-singular point of $\eta$, which proves (\ref{eq:ker_equals_char_foliation}). \end{proof}

\noindent There is also a corresponding description of any characteristic vector field.

\begin{lemma} \label{lem:ker_lambda_description} Fix a contact Hamiltonian manifold $(\Sigma,\eta)$ with contact Hamiltonian form $\lambda$. Then a vector field $Z$ is characteristic if and only if there is a smooth nowhere vanishing function $f$ such that $\iota_Zd\lambda = f\lambda$. \end{lemma}

\begin{proof} Away from the singularities of $\eta$, Lemma \ref{lem:ker_description} immediately implies that $Z$ is characteristic if and only if $\iota_Zd\lambda|_\eta = 0$ or equivalently $\iota_Zd\lambda = f\lambda$ for a (unique) smooth function $f$. Thus, it suffices to prove the result in a neighborhood $U$ of the singularities where $d\lambda$ is symplectic. Fix a volume form $\mu$ such that $n \cdot \mu = \pm d\lambda^n$ in the neighborhood $U$. For this choice of volume form, it is simple to see that for any vector field $Z$, we have
\[
\iota_Z\mu = \pm f \cdot \lambda \wedge d\lambda^{n-1} \text{in $U$} \qquad\text{if and only if}\qquad \iota_Zd\lambda = f \cdot \lambda\text{ in $U$}
\]
Moreover, a vector field $Z$ is characteristic on $U$ if and only if it satisfies the formula on the left for some choice of $f$ that is nowhere vanishing on $U$. This proves the result.\end{proof}

Any hypersurface in a contact Hamiltonian manifold that is transverse to the characteristic foliation has a natural contact structure, and a characteristic flow induces a local identification with the symplectization.

\begin{lemma}[Contact Transversals] \label{lem:contact_transversals} Any hypersurface $\Gamma$ in a contact Hamiltonian manifold $(\Sigma,\eta)$ that is transverse to the characteristic foliation $\Sigma_\eta$ is a contact manifold with contact structure
\[
\xi = \eta \cap T\Gamma
\]
Moreover, if $\Gamma$ is compact then, for small $\epsilon$, the flow by any characteristic vector field $Z$ defines an embedding
\[
[-\epsilon,\epsilon]_s \times \Gamma \to \Sigma \qquad\text{inducing an identification}\qquad \eta = \on{span}(\partial_s) \oplus \xi
\]\end{lemma}

\begin{proof} Take any contact Hamiltonian form $\lambda$ and let $Z$ be a characteristic vector field. By assumption, it is non-zero along $\Gamma$ and transverse to $\Gamma$. Since $Z$ is the unique vector field in the kernel of the interior product with $\lambda \wedge d\lambda^{n-1}$, it follows that $\lambda \wedge d\lambda^{n-1}$ restricts to a nowhere vanishing $(2n-1)$-form on $\Gamma$. Thus $\Gamma$ is contact with contact structure $\on{ker}(\lambda|_\Gamma) = \eta \cap T\Gamma$. Next, let $U$ be a compact neighborhood of $\Gamma$. The flow by $Z$ induces a map $\Phi:[-\epsilon,\epsilon]_s \times U \to \Sigma$. By Lemma \ref{lem:ker_lambda_description}, the flow satisifes
\[
\mathcal{L}_Z\lambda = \iota_Zd\lambda = f\lambda \qquad\text{and thus}\qquad \Phi^*_s\lambda = F_s \lambda\]
for some smooth function $f$ and family of smooth functions $F_s$ on $U$. It follows that the kernel of the 1-form $\alpha = \Phi^*\lambda$ restricted to $[-\epsilon,\epsilon]_s \times \Gamma$ is invariant under translation by $s$. The kernel along $0 \times \Gamma$ is spanned by $\partial_s$ and $\eta \cap T\Gamma = \xi$, so this proves the result. \end{proof}

\begin{cor}[Conformally Liouville] \label{cor:conformally_Liouville_transversal} For any point $P$ in a contact Hamiltonian manifold $(\Sigma,\eta)$, there is a contact Hamiltonian form $\lambda$ for $\eta$ such that
\[
\lambda \text{ is a Liouville form in a neighborhood of $P$}
\]
\end{cor}

\begin{proof} If $P$ is a singular point of $\eta$, then any contact Hamiltonian form is Liouville near $P$ by Definition \ref{def:contact_Hamiltonian_form}. At a non-singular point $P$, the characteristic foliation is non-singular. Choose a small hypersurface $\Gamma$ containing $P$. By Lemma \ref{lem:contact_transversals}, there is a chart containing $P$ in its interior
\[
[-\epsilon,\epsilon] \times \Gamma \to \Sigma \qquad\text{such that}\qquad \eta = \partial_s \oplus \xi
\]
for some contact form $\xi$ on $\Gamma$. If $\alpha$ is a contact form on $\Gamma$ for $\xi$, then $\lambda = e^s \alpha$ is a contact Hamiltonian form in the chart that is Liouville.\end{proof}

\begin{remark} The observations in Lemmas \ref{lem:ker_description}-\ref{lem:contact_transversals} show that the flow of the characteristic foliation can be viewed as conformally symplectic. More specifically, fix a contact Hamiltonian manifold $(\Sigma,\eta)$ with contact Hamiltonian form $\lambda$ and a characteristic vector field $Z$. Also fix two open sets $U$ and $W$ such that the time $T$ flow $\Phi$ of $Z$ maps $U$ to $W$. Then $\Phi$ is a map
\[
\Phi:U \to W\qquad\text{with}\qquad \Phi^*\lambda|_W = F\cdot \lambda|_U
\]
Alternatively, $\Phi$ is conformally symplectic (or more appropriately, conformally Liouville) whenever $\lambda$ is chosen to be Liouvile on $U$ and $W$. If $U$ and $W$ are replaced with transverse hypersurfaces, then $\Phi$ defines a contactomorphism between those transversals. From this perspective, characteristic flows are the flow analogues of contactomorphisms.\end{remark}

\begin{figure}[h]
    \centering
    \includegraphics[width=\linewidth]{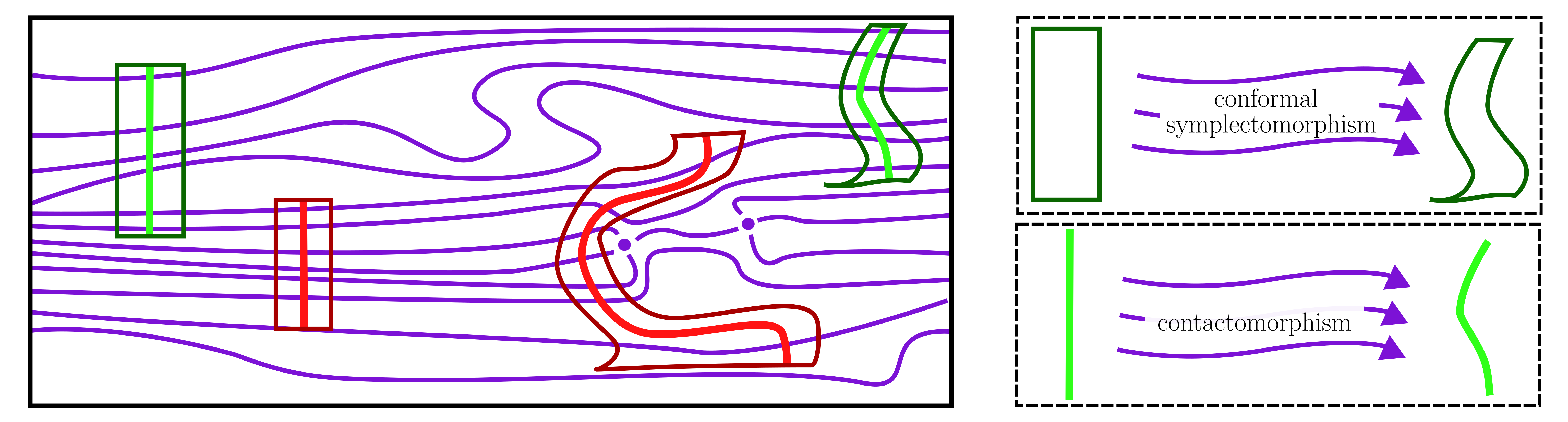}
    \caption{The characteristic foliation along with contact transversals mapped to each other, and small neighborhoods isomorphic to their symplectizations.}
    \label{fig:characteristic_flow}
\end{figure}

There is a reasonable notion of isotropic and Legendrian sub-manifolds in the setting of contact Hamiltonian manifolds.

\begin{definition}[Isotropic] An embedded (or more generally, immersed) sub-manifold $\Lambda \to \Sigma$ in a contact Hamiltonian manifold $(\Sigma,\eta)$ is \emph{isotropic} if $T\Lambda \subset \eta$.\end{definition}

\begin{lemma}[Dimension] \label{lem:isotropic_dimension} Any isotropic sub-manifold $\Lambda$ in a contact Hamiltonian manifold $(\Sigma,\eta)$ has
\begin{equation}\label{eq:isotropic_dim}
\on{dim}(\Lambda) \le \frac{1}{2} \on{dim}(\Sigma)
\end{equation}
The isotropic $\Lambda$ is tangent to the characteristic foliation $\Sigma_\eta$ (or equivalently, invariant under any characterstic flow) if the inequality (\ref{eq:isotropic_dim}) is an equality.\end{lemma}

\begin{proof} Let $\lambda$ be any contact Hamiltonian form for $\eta$ and note that $d\lambda|_{T\Lambda} = 0$. There are two cases. At a point where $\eta = T\Sigma$, the form $d\lambda$ is symplectic and $T\Lambda$ is isotropic, implying (\ref{eq:isotropic_dim}). Since $\Sigma_\eta$ has a singularity at such a point, the tangency condition holds tautologically.

\vspace{3pt}

At any other point, we know that $d\lambda|_\eta$ is maximally non-degenerate and $T\Lambda \subset \eta$. It follows that $d\lambda|_\eta$ descends to a symplectic form on the quotient $\eta/\Sigma_\eta$ and $T\Lambda$ projects to an isotropic subspace  $V \subset \eta/\Sigma_\eta$. It follows that at that point
\[
\on{dim}(T\Lambda) \le \on{dim}(V) + 1 = \frac{1}{2}\on{dim}(\eta/\Sigma_\eta) + 1 = \frac{1}{2}\on{dim}(\Sigma)
\]
This can only be an equality if $\on{dim}(T\Lambda) = \on{dim}(V) + 1$ which implies that $\Sigma_\eta \subset T\Lambda$. \end{proof}

\begin{definition}[Legendrian] A \emph{Legendrian} sub-manifold $\Lambda \to \Sigma$ is an isotropic sub-manifold of maximal dimension.
\end{definition}

\subsection{Contactization} Any contact Hamiltonian manifold admits a contactization, which is a contact structure on the product of the manifold with a small interval. This construction requires the following choice of auxilliary data.

\begin{definition}[Framing] \label{def:framing} A \emph{framing} $(u,\theta)$ for a contact Hamiltonian form $\lambda$ on a contact Hamiltonian manifold $(\Sigma,\eta)$ is a pair of a smooth function $u$ and a smooth 1-form $\theta$ on $\Sigma$ where
\begin{equation} \label{eq:framing_volume_form}
u \cdot d\lambda^n + n\theta \wedge \lambda \wedge d\lambda^{n-1} \quad\text{is a volume form on $\Sigma$}
\end{equation}\end{definition}

\begin{definition}[Contactization] The \emph{contactization} $(C\Sigma,\xi_\eta)$ of a contact Hamiltonian manifold $(\Sigma,\eta)$ is the contact manifold
\[
C\Sigma = [-\epsilon,\epsilon]_s \times \Sigma
\]
equipped with the contact structure $\xi$ defined by the contact form
\[
\alpha = uds + s(\theta + du) + \lambda \qquad\text{for any contact Hamiltonian form $\lambda$ and framing $(u,\theta)$}
\]\end{definition}

To show that the contactization is a contact manifold independent of the framing, we need some elementary results. First, note that there is an ample supply of framings available.

\begin{lemma}[Volume Form] \label{lem:existence_of_framing} For any contact Hamiltonian from $\lambda$ and volume form $\mu$ on $\Sigma$, there is a framing $(u,\theta)$ such that
\[
\mu = u \cdot d\lambda^n + n\theta \wedge \lambda \wedge d\lambda^{n-1}
\]
\end{lemma}

\begin{proof} It follows from Definition \ref{def:contact_Hamiltonian_manifold} that there is a cover of $\Sigma$ by two open sets $U$ and $V$ such that
\begin{equation} \label{eq:existence_framing}
d\lambda^n \neq 0 \text{ on $U$}\qquad\text{and}\qquad \lambda \wedge d\lambda^{n-1} \neq 0 \text{ on $V$}
\end{equation}
Fix an auxilliary volume form $\mu$ and choose a partition of unity subordinate to the cover by $U$ and $V$. In other words, choose two smooth and non-negative functions $F_U$ and $F_V$ such that
\[
\on{supp}(F_U) \subset U \qquad \on{supp}(F_V) \subset V \qquad F_U + F_V = 1
\]
Then (\ref{eq:existence_framing}) implies that we may choose a smooth function $u$ compactly supported in $U$ and a smooth 1-form $\theta$ compactly supported in $V$ such that
\[
u \cdot d\lambda^n = F_U \cdot \mu\qquad\text{and}\qquad \theta \wedge \lambda \wedge d\lambda^{n-1} = F_V \cdot \mu 
\]
For this choice of $u$ and $\theta$, the form (\ref{eq:framing_volume_form}) coincides with $\mu$. This proves the claim. \end{proof}

\begin{remark} \label{rmk:convexity_of_framing} Note that the space of choices of framing is contractible (and even convex). This is immediate from the linearity of (\ref{eq:framing_volume_form}) in the framing. \end{remark}

We then have the following construction of an (essentially unique) germ of a contact structure on the thickening of a contact Hamiltonian manifold.

\begin{lemma}[Contact Germ] Fix a contact Hamiltonian form $\lambda$ and a framing $(u,\theta)$ on a contact Hamiltonian manifold $(\Sigma,\eta)$. Then for sufficiently small $\epsilon$
\[
\alpha = uds + s(\theta + du) + \lambda \qquad\text{is a contact form on}\qquad [-\epsilon,\epsilon]_s \times \Sigma
\]
The corresponding contact structure $\xi$ is independent of the contact Hamiltonian form $\lambda$ and the framing $(u,\theta)$ in a neighborhood of $\Sigma = 0 \times \Sigma$, up to ambient isotopy fixing $0 \times \Sigma$. \end{lemma}

\begin{proof} The associated differential form in the top degree associated to $\alpha$ can be computed as
\[
\alpha \wedge d\alpha^n = ds \wedge (u \cdot d\lambda^n + n\theta \wedge \lambda \wedge d\lambda^{n-1}) \qquad\text{along}\qquad 0 \times \Sigma \subset \R_s \times \Sigma
\]
This is a volume form in a small neighborhood of $0 \times \Sigma$ if and only if $(u,\theta)$ is framing for $\lambda$. Therefore $\alpha$ is a contact form near $0 \times \Sigma$ for any framing. For the uniqueness, note that
\[
\xi \cap T\Sigma = \eta \qquad\text{for the contact structure $\xi$ for any choice of $\lambda,u$ and $\theta$}
\]
Note that this is an equality of singular hyperplane fields. The uniqueness result thus follows from Lemma \ref{lem:stability_for_hypersurfaces} below. \end{proof}

\begin{lemma} \cite{g2008intro} \label{lem:stability_for_hypersurfaces} Let $\xi$ and $\xi'$ be a pair of contact structures on $Y$ and let $\Sigma \subset Y$ be a hypersurface such that $\xi \cap T\Sigma = \xi' \cap T\Sigma$ as singular hyperplane fields. Then there is an ambient isotopy
\[\Phi:[0,1]_t \times Y \to Y \qquad\text{such that}\qquad  \Phi_0 = \on{Id} \qquad \Phi_1^*\xi = \xi'\text{ near $\Sigma$} \quad\text{and}\quad \Phi_t(\Sigma) = \Sigma\]\end{lemma}

\noindent Finally, note that the contactization provides a standard neighborhood model for any hypersurface in a contact manifold.

\begin{lemma}[Standard Neighborhood] \label{lem:std_nbhd} Let $\Sigma$ be a closed hypersurface in a contact manifold $(Y,\xi)$ with the contact Hamiltonian structure $\eta = \xi \cap T\Sigma$. Then there is a contact embedding
\[
\iota:C\Sigma \to Y \qquad\text{with}\qquad \iota(0 \times \Sigma) = \Sigma
\]
\end{lemma}

\begin{proof} Choose a tubular neighborhood $\iota:[-\epsilon,\epsilon] \times \Sigma \to Y$ of $\Sigma$ and a contact form $\beta$ on $Y$. Let $\lambda$ be the restriction of $\iota^*\beta$ to $0 \times \Sigma$ and let $\alpha$ be the contactization contact form corresponding to an arbitrary framing. Then $\alpha$ and the pullback $\iota^*\beta$ are two contact forms on $C\Sigma$ that pull back to $\lambda$ on $0 \times \Sigma$. Thus $\iota$ can be isotoped to a contact embedding with $\iota(0 \times \Sigma) = \Sigma$ by Lemma \ref{lem:stability_for_hypersurfaces}.
\end{proof}

\subsection{Liouville Subsets} \label{subsec:Lioiville_subsets} We have already seen that Liouville domains are a natural example of contact Hamiltonian manifolds. While most contact Hamiltonian manifolds are not Liouville, even up to a conformal factor, it is still useful to consider the subsets that are.

\begin{definition}[Liouville Subsets] A subset $\Lambda \subset \Sigma$ in a contact Hamiltonian manifold $(\Sigma,\eta)$ is called \emph{Liouville} if there is a contact Hamiltonian form $\lambda$ such that
\[
\text{$\lambda|_U$ is a Liouville form on a neighborhood $U$ of $\Lambda$}
\]
We say that $\Lambda$ is \emph{positive} or \emph{negative} if $\lambda|_U$ is positive or negative Liouville (with respect to a choice of orientation or volume form). \end{definition}

\begin{remark} Note that a subset can be both positive and negative Liouville.
\end{remark}

There is a characterization of Liouville subsets that only makes reference to the smooth dynamics of the characteristic vector field. Specifically, we have the following lemma.

\begin{lemma}[Divergence Criterion] Fix a subset $\Lambda \subset \Sigma$ of a contact Hamiltonian manifold and a volume form $\mu$. Then $\Lambda$ is positive Liouville if and only if there is a characteristic vector field $Z$ such that
\[
\on{div}(Z,\mu)(Z,\mu) > 0 \qquad\text{along $\Lambda$}
\]
Similarly, $\Lambda$ is negative Liouville if and only if there is a characteristic vector field with negative divergence.\end{lemma}

\begin{proof} Note that $d\lambda^n = F \cdot \mu$, and that $\lambda$ is positive (or negative) Liouville if and only if $F$ is positive (or negative) everywhere along $\Lambda$. Now suppose that $Z$ is the characteristic vector field determined by $\mu$ and a contact Hamiltonian form $\lambda$ as in (\ref{eq:characteristic_vectorfield}). Then \begin{equation} \label{eq:divergence_equation}
d\lambda^n =d(\lambda \wedge d\lambda^{n-1}) = d(\iota_Z \mu) = \on{div}(Z,\mu)(Z,\mu) \cdot \mu  
\end{equation}
Thus, a characteristic vector field $Z$ has positive divergence if and only if the corresponding contact Hamiltonian form $\lambda$ is positive Liouville.  \end{proof}

The basic examples of Liouville subsets that will be used later are the following.

\begin{example}[Transversal] \label{ex:transversal_Liouville} Fix a contact sub-manifold $P \subset \Sigma$ in $(\Sigma,\eta)$ transverse to the characteristic foliation. Then Lemma \ref{lem:contact_transversals} and Corollary \ref{cor:conformally_Liouville_transversal} imply that there is an embedding
\[\R_s \times P \to \Sigma \qquad \text{of symplectization of $P$}\]
preserving the contact Hamiltonian structures. Therefore $P$ is (positive and negative) Liouville.
\end{example}

\begin{example}[Singularities] Any singularity $x$ of the characteristic foliation of a contact Hamiltonian manifold $(\Sigma,\eta)$ is automatically Liouville, essentially by Definition \ref{def:contact_Hamiltonian_form}.
\end{example}

\begin{lemma}[Hyperbolic Orbits] Let $\gamma$ be a hyperbolic orbit of the characteristic foliation of a contact Hamiltonian manifold $(\Sigma,\eta)$. Then $\gamma$ is Liouville. 
\end{lemma}

\begin{proof} We leave this as an exercise to the reader (cf. Breen \cite{b2021} for help).
\end{proof}

\begin{remark}[Hyperbolic Singularities And Orbits] Recall that a fixed point $x$ of a diffeomorphism $\Phi$ of a smooth manifold $M$ is \emph{hyperbolic} if the differential 
\[
T_x\Phi:T_xM \to T_xM \qquad\text{at the point $x$}
\]
is a hyperbolic linear map, e.g. an invertible linear map with no eigenvalues of unit norm. Similarly, a singularity $x$ of a singular line field $L$ is hyperbolic if it is a hyperbolic fixed point of map $\Phi_T$ for $T > 0$, where $\Phi$ is any generating flow. Finally, a closed orbit $\gamma$ of a vector field $V$ is hyperbolic if the closed orbit corresponds to a hyperbolic fixed point of the Poincar\'{e} return map on a small hypersurface transverse to the orbit at a point. \end{remark}

\begin{remark} One interpretation of Example \ref{ex:transversal_Liouville} is the following: regions where the characteristic foliation has trivial (non-recurrent) dynamics are automatically both positive and negative Liouville. In contrast, we will shortly see that regions containing recurrent dynamics (e.g. periodic orbits) are generally either one or the other. \end{remark}

\subsection{Stable And Unstable Manifolds} \label{subsec:stable_unstable_manifolds} We next analyze the stable and unstable manifolds of closed orbits and singularities. We recall the definition.

\begin{definition}[Stable/Unstable Manifolds] The \emph{stable manifold} and \emph{unstable manifold} of a closed orbit $\gamma$ of a singular line field is the sets given by
\[
W^s(\gamma) = \big\{x  \; : \; \lim_{T \to \infty}\big(\on{dist}(\Phi_T(x),\gamma)\big) \to 0\big\}
\]
\[
W^u(\gamma) = \big\{x \; : \; \lim_{T \to -\infty}\big(\on{dist}(\Phi_T(x),\gamma)\big) \to 0\big\}
\]
where $\Phi$ is generated by a smooth vector field that spans the singular line field and $\on{dist}(\cdot)$ is the distance for an arbitrary Riemannian metric. \end{definition}

\noindent The standard Stable Manifold Theorem (cf. \cite[Thm 6.1.1]{fh2019}) states that the stable (and unstable) manifolds are indeed smooth manifolds when the orbit is a hyperbolic orbit or singularity.

\begin{definition}[Index] \label{def:index} The \emph{index} $\on{ind}(\gamma)$ of a hyperbolic closed orbit (or singularity) of an oriented singular line field is the dimension of the stable manifold of $\gamma$.
\[
\on{ind}(\gamma) = \on{dim}(W^s(\gamma)) - 1 \text{ if $\gamma$ is an orbit} \qquad\text{and}\qquad \on{ind}(\gamma) = \on{dim}(W^s(\gamma)) \text{ if $\gamma$ is a singularity} 
\]
\end{definition}

\begin{lemma}[Isotropic Stable/Unstable Manifolds] \label{lem:isotropic_stable} Let $\gamma$ be a closed hyperbolic orbit (or singularity) of the characteristic foliation of a contact Hamiltonian manifold $(\Sigma,\eta)$. Then
\[
W^s(\gamma) \text{ is isotropic if $\gamma$ is positive Liouville}
\qquad\text{and}\qquad
W^u(\gamma) \text{ is isotropic if $\gamma$ is negative Liouville}
\]
\end{lemma}

\begin{proof} By reversing the orientation (and thus the orientation of the characteristic foliation), we can reduce to the case of the stable manifold of a positive Liouville hyperbolic orbit (or singularity). 

\vspace{3pt}

Fix a point $x$ in $W^s(\gamma)$ and a positive Liouville form $\lambda$ near $\gamma$. Let $Z$ be a characteristic vector field that is the Liouville vector field of $\lambda$ in a neighborhood $U$ of $x$ and let $\Phi$ be the corresponding characteristic flow. Since $\Phi$ preserves isotropics, we can apply the flow $\Phi$ for large time and assume that $p$ is in a neighborhood where $\lambda$ is positive Liouville. Then
\[
\mathcal{L}_Z\lambda = \lambda \qquad\text{and}\qquad \Phi_T^*\lambda = e^T \cdot \lambda
\]
On the other hand, if $v \in TW^s(\gamma)$ is a vector tangent to $W^s(\gamma)$ at $x$, then $|T\Phi_T(v)|$ is uniformly bounded for all $T > 0$. It follows that there is a constant $C$ such that
\[
|\lambda(v)| = e^{-T} \cdot |\Phi^*_T\lambda(v)| = e^{-T} \cdot |\lambda(T\Phi_T(v))| \le Ce^{-T}
\]
Taking the limit as $T \to \infty$ yields the desired result.\end{proof}

\noindent As a corollary, we acquire the following lemmas of Breen-Honda-Huang \cite{hh2019,b2021}. They state that hyperbolic closed orbits and singularities are either positive or negative Liouville, but not both. 

\begin{lemma}[Liouville Orbits] \label{lem:breen_divergence} Let $\gamma$ be a hyperbolic closed orbit of the characteristic foliation of a contact Hamiltonian $2n$-manifold $(\Sigma,\eta)$. Then $\gamma$ is either
\[\text{positive Liouville with $\on{ind}(\gamma) \le n-1$} \qquad\text{or}\qquad \text{negative Liouville with $\on{ind}(\gamma) \ge n$}\]
\end{lemma}

\begin{lemma}[Liouville Singularities] \label{lem:hh_divergence} Let $x$ be a hyperbolic singularity of the characteristic foliation of a contact Hamiltonian $2n$-manifold $(\Sigma,\eta)$. Then $x$ is either
\[\text{positive Liouville with $\on{ind}(x) \le n$} \qquad\text{or}\qquad \text{negative Liouville with $\on{ind}(x) \ge n$}\]
The sign (positive or negative) of a singularity can be computed via the eigenvalues of the linearization.\end{lemma}

\noindent The two lemmas follow immediately from Lemma \ref{lem:isotropic_stable}, Lemma \ref{lem:isotropic_dimension} and Definition \ref{def:index}.

\begin{figure}[h]
    \centering
    \includegraphics[width=.8\linewidth]{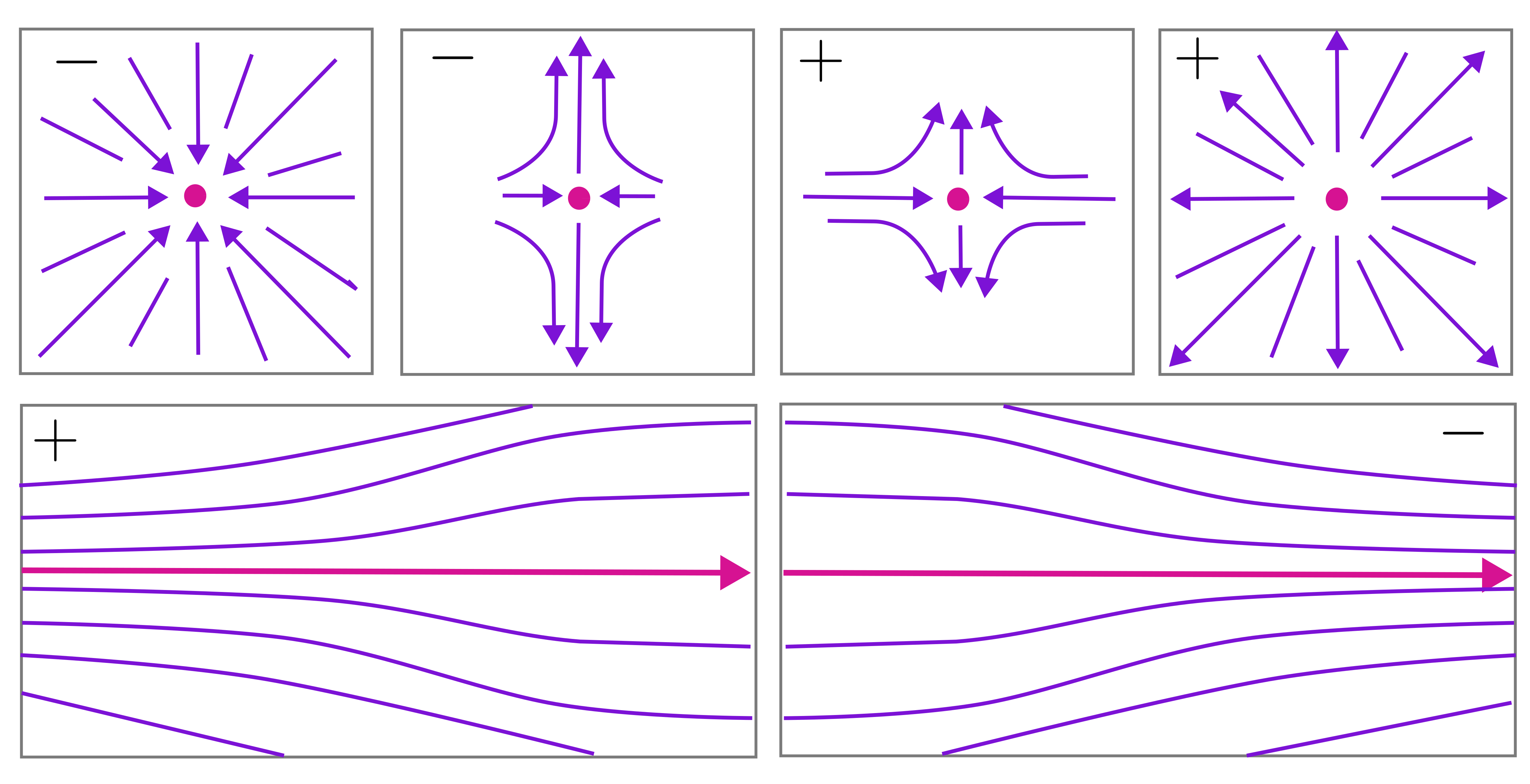}
    \caption{The different types of singularities and closed orbits in dimension two. The index one singularities are distinguished by the sign of the divergence.}
    \label{fig:critical_point_examples}
\end{figure}

\begin{remark}[Symplectic Case] In the case of symplectic diffeomorphisms and flows on a symplectic manifold $(X,\Omega)$, any hyperbolic fixed point or closed orbit must have index given by half the dimension of $X$. This follows immediately from the corresponding statement for hyperbolic symplectic matrices. In the conformally symplectic setting, there is no such restriction.
\end{remark}

\subsection{No Anosov Characteristic Foliations} \label{subsec:no_anosovs} So far, we have introduced a special class of conformally symplectic dynamical system: characteristic foliations (or flows) on contact Hamiltonian manifolds. Here we discuss a non-trivial constraint on the dynamics of characteristic foliations. Precisely, we prove that characteristic flows cannot be Anosov, generalizing a recent result of Asaoka-Mitsumatsu \cite{asaoka2025there} for contactomorphisms. 

\begin{thm}[No Anosovs] \label{thm:no_Anosovs} Let $\Phi$ be a characteristic flow of a closed contact Hamiltonian manifold $(\Sigma,\eta)$. Then $\Phi$ is not an Anosov flow.
\end{thm}

\noindent Anosov flows are a fundamental class of smooth dynamical system that include geodesic flows on hyperbolic manifolds and suspensions of hyperbolic torus automorphisms. They have been central objects of study in dynamics since the work of Anosov \cite{a1963} and Smale \cite{smale1967differentiable}, and thus their exclusion from characteristic flows is indeed a major constraint.

\vspace{3pt}

For the proof of Theorem \ref{thm:no_Anosovs}, we will need to recall some basic facts about Anosov flows. For a more complete reference, we refer the reader to the monograph of Fisher-Hasselblatt \cite{fh2019}. 

\begin{definition}[Anosov] A smooth flow $\Phi$ generated by a vector field $V$ on a closed manifold $M$ is \emph{Anosov} if there are $\Phi$-invariant continuous sub-bundles
\[
E^s(\Phi) \subset TM \qquad\text{and}\qquad E^u(\Phi) \subset TM
\]
called the \emph{stable bundle} and the \emph{unstable bundle}, such that the tangent bundle splits as
\[
TM = \on{span}(V) \oplus E^s(\Phi) \oplus E^u(\Phi)
\]
and such that there are constants $A,B > 0$ such that, for all sufficiently large $T > 0$, we have
\[
|T\Phi_T(v)| \le A\exp(-BT) \text{ for }v \in E^s(\Phi) \qquad\text{and}\qquad |T\Phi_{-T}(v)| \le A\exp(-BT) \text{ for }v \in E^u(\Phi)
\]\end{definition}

\begin{definition}[Weak-(Un)stable Foliations] The \emph{weak-stable foliation} $W^s(\Phi)$ and the \emph{weak-unstable foliation} $W^u(\Phi)$ are the continuous foliations with smooth leaves tangent to the distributions
\[
\on{span}(V) \oplus E^s(\Phi) \qquad\text{and}\qquad  \on{span}(V) \oplus E^u(\Phi)
\]
We denote the leaves of the weak-stable and weak-unstable foliations containing a point $p$ by
\[
W^s(\Phi,p) \qquad\text{and}\qquad W^u(\Phi,p)
\]
\end{definition}

\begin{remark} The stable and unstable bundles, and the corresponding weak-stable and weak-unstable foliations, of an Anosov flow are unique.
\end{remark}

\begin{lemma}[Invariant Bundles] \label{lem:invariant_bundles} Let $\Phi$ be an Anosov flow on a manifold $M$ and let $F \subset TM$ be a $\Phi$-invariant $k$-plane field. Then
\[
E^s(\Phi) \subset F \text{ at a point $p$}\qquad\text{if and only if}\qquad E^s(\Phi) \subset F \text{ along all of $W^s(\Phi,p)$}
\]
Then $E^s(\Phi) \subset F$ along the entire weak-stable leaf $W^s(\Phi,p)$ containing $p$. \end{lemma}

\begin{exercise} Prove Lemma \ref{lem:invariant_bundles} following the proof for Anosov diffeomorphisms given in \cite{asaoka2025there}.
\end{exercise}

\begin{proof} (Theorem \ref{thm:no_Anosovs}) Suppose for contradiction that $\Phi$ is an Anosov characteristic flow on $(\Sigma,\eta)$. Fix a contact Hamiltonian form $\lambda$ and a volume form $\mu$ such that
\[
\iota_Z\mu = \lambda \wedge d\lambda^{n-1}
\]
Since $\Phi$ is characteristic, there is a smooth family of functions
\[
F:\R \times \Sigma \to \R \qquad\text{such that}\qquad \Phi^*_T\lambda = F_T \cdot \lambda
\]
We assume (after possibly reversing orientations) that the leaves of the stable foliation are more than half the dimension of $\Sigma$. Consider the set
\[
S = \{p \in \Sigma \; : \; E^s(\Phi) \subset \eta \text{ at the point }p\}
\]
We claim that $S$ is non-empty. Indeed, suppose otherwise. Then at every point $p$, we can choose a vector $W \in E^s(\Phi)$ at $p$ such that $\lambda(W) = 1$. Then we compute that
\[
F_T(p) = F_T(p) \cdot |\lambda_p(W)| = |\Phi^*_T\lambda(W)_p| = |\lambda(T\Phi_T(W))_p| \le |\lambda_p| \cdot |T\Phi_T(W)_p| \le A\exp(-BT)
\]
for some $A,B > 0$. Since $Z$ generates the flow $\Phi$, we can compute that
\[
\iota_Z(\Phi_T^*\mu) = \Phi_T^*(\iota_Z\mu) = \Phi^*_T(\lambda \wedge d\lambda^n) = F_T^n \cdot \lambda \wedge d\lambda^n 
\]
On the other hand, $Z$ is nowhere vanishing by the Anosov assumption. Therefore the interior product with $Z$ is a linear isomorphism from volume forms to multiples of $\lambda \wedge d\lambda^{n-1}$. It follows that we must have
\[
\Phi_T^*\mu = F_T^n \cdot \mu \qquad\text{and thus}\qquad \int_\Sigma \mu = \int_\Sigma \Phi_T^*\mu \le A^n \exp(-nBT) \cdot \int_\Sigma \mu
\]
For sufficiently large $T$, this is clearly impossible. Thus $S$ must be non-empty.

\vspace{3pt}

Finally, take any point $p$ in $S$. By Lemma \ref{lem:invariant_bundles}, the entire stable leaf $W^s(\Phi,p)$ is tangent to $\eta$. In particular, $W^s(\Phi,p)$ is isotropic. However, the dimension must then be half of that of $\Sigma$ or less. This is a contradiction.\end{proof}

\section{Convexity And Characteristic Dynamics} \label{sec:convexity}

A contact Hamiltonian manifold is convex if the contactization admits a translation invariant contact structure. More precisely, we adopt the following definition.

\begin{definition}[Convex Structure] \label{def:convex_structure} A contact Hamiltonian manifold $(\Sigma,\eta)$ is \emph{convex} if there is a smooth function $u$ such that $(u,-du)$ is a framing, or equivalently
\[
u ds + \lambda \qquad\text{is a contact form on $C\Sigma = [-\epsilon,\epsilon] \times \Sigma$}
\]
We refer to the function $u$ as a \emph{framing function} for the (convex) contact Hamiltonian manifold.\end{definition}

Convexity is the direct analogue of stability for Hamiltonian manifolds and hypersurfaces in symplectic manifolds \cite{ce2015}. Alternatively, Definition \ref{def:convex_structure} is simply the intrinsic formulation of convexity for hypersurfaces in contact manifolds, in the sense of Giroux \cite{g1991}.

\begin{definition}[Convex Hypersurface] \label{def:convex_hypersurface} A hypersurface $\Sigma$ in a contact manifold $(Y,\xi)$ is \emph{convex} if there exists a contact vector field $V$ that is transverse to $\Sigma$.
\end{definition}

\noindent Convex hypersurfaces are a fundamental tool in contact topology. In this section, we discuss recent breakthroughs in the theory of convexity from a dynamical perspective, partly using the intrinsic language of contact Hamiltonian manifolds. These results are variously due to Giroux \cite{g1991,giroux2002geometrie}, Honda-Huang \cite{hh2018,hh2019}, Breen \cite{b2024,b2021}, Breen-Honda-Huang \cite{bhh2023} and the author \cite{jc2024}. 

\subsection{History And Motivation} Although the definition of convex hypersurfaces is fairly natural, their key role in contact topology may not be completely obvious at first glance to the uninitiated reader. Let us take a brief motivational and historical detour before proceeding further.

\vspace{3pt}

Convex hypersurfaces were introduced by Giroux \cite{g1991,giroux2002geometrie} in the early 1990s, who established many of their fundamental properties in dimension three. Giroux was originally motivated by the notion of a convex contact manifold introduced by Eliashberg-Gromov \cite[Def 3.5.A]{eliashberg1991convex}. This is a contact manifold that admits a Morse function that is contact in the following sense (cf. \cite{hh2019,sackel2019getting}).

\begin{definition}[Contact Morse] A Morse function $F$ on a contact manifold $(Y,\xi)$ is \emph{contact} if there is a gradient-like vector field $V$ that is contact (i.e. that preserves the contact structure).
\end{definition}

\noindent Morse theory plays a fundamental role in smooth topology, and thus it is clearly desirable to have a well-developed Morse theory for contact manifolds. However, it was originally not even clear that contact Morse functions exist on any contact manifold. This problem was posed explicitly by Eliashberg-Gromov in \cite[p. 34]{eliashberg1991convex}.

\begin{question} Does every closed contact manifold have a contact Morse function?
\end{question}

\noindent In \cite{g1991}, Giroux resolved this problem in dimension three using convex surface theory.

\begin{thm}[Giroux] \label{thm:existence_contact_Morse_dim_3} Every closed contact 3-manifold $(Y,\xi)$ admits a contact Morse function.
\end{thm}

\noindent In smooth topology, a Morse function $F$ is equivalent to a handlebody decomposition, which itself may be viewed as a 1-parameter family of manifolds, the regular level sets $\Sigma(a) = F^{-1}(a)$, with a discrete set of simple changes (surgeries coming from handle attachments) that occur when the parameter crosses a critical value. It is simple to see that
\[\Sigma(a) \subset (Y,\xi) \text{ is a convex hypersurface if $F$ is a contact Morse function}
\]
The key observation of Giroux was that it is actually simpler to construct this family of convex surfaces in dimension three, rather than constructing the contact Morse function directly. We will discuss Giroux's proof method later in this section.

\vspace{3pt}

\begin{figure}[h]
    \centering
    \includegraphics[width=\linewidth]{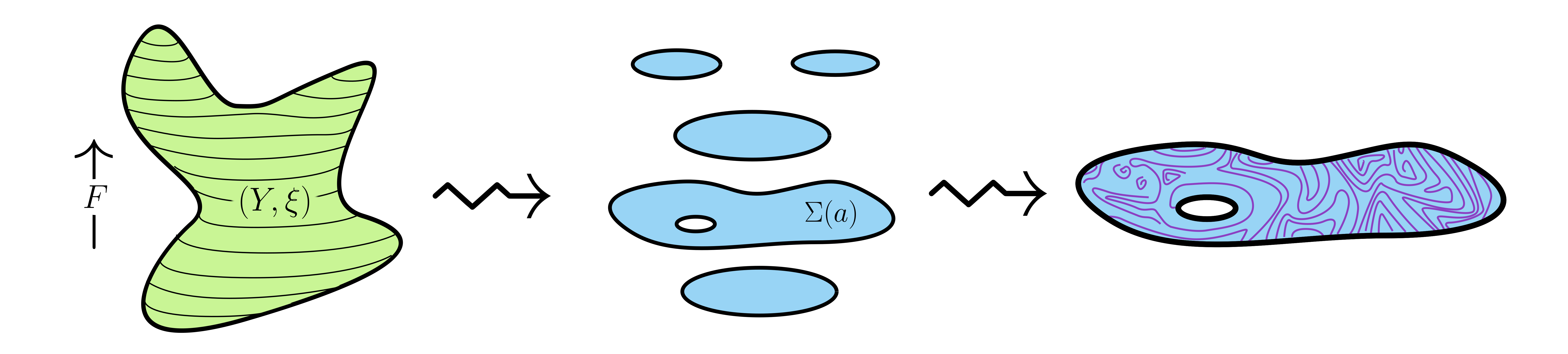}
    \caption{Giroux's strategy for constructing contact Morse functions used an analysis of the family of level sets given by a Morse function and specifically the analysis (and modification) of their characteristic foliations. }
    \label{fig:contact_Hamiltonian_examples}
\end{figure}

Following the work of Giroux, convex surfaces developed into the fundamental tool in the theory of contact 3-manifolds. Landmark applications include the Giroux correspondence between open  book decompsotions and contact structures in dimension three due to Giroux \cite{giroux2002geometrie}; the finiteness of torsion free contact structures due to Colin-Giroux-Honda \cite{cgh2009}; the discovery of 3-manifolds with no tight contact structures \cite{e2001} and tight contact structures with no fillings \cite{eh2002} by Honda-Etnyre; and solutions to many classification problems \cite{h1999,h2000,y1997,cm2020,cgh2003}. We will review many of Giroux's key results later in this section. 

\vspace{3pt}

In dimensions five and higher, convex hypersurfaces were quite poorly understood until the recent, groundbreaking work of Honda-Huang \cite{hh2018,hh2019} who established very general existence and approximation results for convex hypersurfaces. An immediate application of their work was a proof of Theorem \ref{thm:existence_contact_Morse_dim_3} in any dimension via a higher dimensional adaptation of Giroux's strategy. Their work also lead to further applications, such as a well developed handlebody theory for convex hypersurfaces, including a higher dimensional theory of bypasses, and the completed proof of the Giroux correspondence between open book decompositions and contact structures in all dimensions by Breen-Honda-Huang \cite{bhh2023}. Due to space constraints, we will focus on only a small part of this story. However, we encourage the reader to investigate these topics further in the references \cite{hh2018,hh2019,bhh2023,b2021}.

\subsection{Dividing Set And Liouville Halves} We now return to a more formal discussion of convex contact Hamiltonian manifolds (or equivalently, convex hypersurfaces). A key feature of convex hypersurfaces is the existence of a dividing set.

\begin{definition}[Dividing Set] The \emph{dividing set} $\Gamma$ of a convex contact Hamiltonian manifold $(\Sigma,\eta)$ with respect to a framing function $u$ is the hypersurface
\[
\Gamma = \{u = 0\} \subset \Sigma
\]
The regions $\Sigma_+$ and $\Sigma_-$ where $u$ is non-negative and non-positive are respectively called the \emph{positive} and \emph{negative} halves of $\Sigma$. We have a splitting
\[
\Sigma = \Sigma_+ \cup_\Gamma \Sigma_-
\]\end{definition}

\begin{figure}[h]
    \centering
    \includegraphics[width=.7\linewidth]{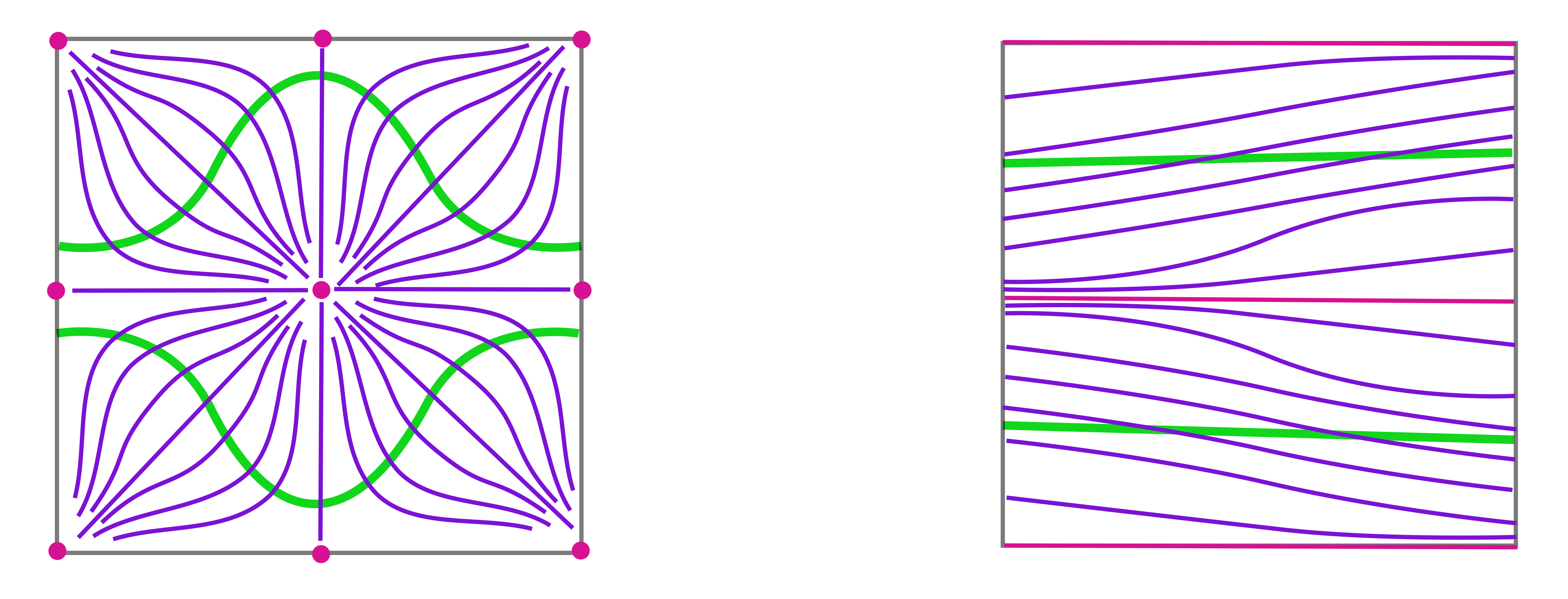}
    \caption{Characterstic foliations on the torus with dividing sets in green. }
    \label{fig:dividing_sets}
\end{figure}

\begin{lemma}[Dividing Set Is Contact] The dividing set $\Gamma \subset \Sigma$ is a contact sub-manifold transverse to the characteristic foliation that is independent of the framing function up to isotopy.
\end{lemma}

\begin{proof} Let $u$ be any framing function. By Definition \ref{def:framing}, the differential form (\ref{eq:framing_volume_form}) is a volume form $\mu$. It follows that
\[
du \neq 0 \qquad\text{along the subset}\qquad \Gamma = u^{-1}(0)
\]
Therefore $\Gamma$ is a transversely cut out hypersurface. In order for $du \wedge \lambda \wedge d\lambda^{n-1}$ to be a volume form, it must be the case that $du(Z) \neq 0$ along $\Gamma$ where $Z$ is the associated characteristic vector field to the volume form $\mu$. It follows that $Z$ is transverse to $\Gamma$. The independence of $u$ up to isotopy follows from the contractibility of the space of framing functions (see Remark \ref{rmk:convexity_of_framing}). \end{proof}

\begin{lemma}[Halves Are Liouville] \label{lem:liouville_halves} The positive and negative halves $\Sigma_+$ and $\Sigma_-$ of a convex contact Hamiltonian manifold $(\Sigma,\eta)$ are conformally Liouville domains.
\end{lemma} 

\begin{proof} We prove the result for $\Sigma_+$, as the proof is analogous for $\Sigma_-$. Note that $\Sigma_+$ is isomorphic as a contact Hamiltonian manifold to $S = \Phi_T(\Sigma_+)$ where $\Phi_T$ is the time $T$ characteristic flow. So it suffices to prove that $S$ is conformally Liouville. For this, note that if we choos $T < 0$, then $S \subset \Sigma_+$ is disjoint from the dividing set $\Gamma = u^{-1}(0)$ and we can rescale $u ds + \lambda$ by a positive smooth function so that
\[
u = 1 \text{ on $S$}\qquad\text{and}\qquad u ds + \lambda = ds + \lambda \text{ on $S$}
\]
On the other hand, $ds + \lambda$ is contact if and only if $\lambda$ is a Liouville form. Moreover, the characteristic vector field on $\Sigma$, which agrees with the Liouville vector field of $S = \Sigma_+$, points outward along $\Gamma = \partial\Sigma_+$. Thus $S = \Sigma_+$ is (conformally) a Liouville domain.\end{proof} 

\subsection{Morse-Smale Criterion} \label{subsec:MS_criterion} A basic problem is to determine if a given contact Hamiltonian manifold is convex. As it happens, the most useful criteria that imply convexity are dynamical in nature, and this is a key interaction point between convex hypersurfaces and dynamical systems. Here we discuss the original criterion observed by Giroux \cite{g1991} and generalized by Breen \cite{b2021}. 

\vspace{3pt}

Recall that a \emph{wandering point} $x$ with respect to a flow $\Phi$ is a point such there is a neighborhood $U$ of $x$ whose image under the flow is disjoint from the original open set, for large times. That is, there is a time $S$ such that
\[\Phi_T(U) \cap U = \emptyset \qquad\text{for all}\qquad T \ge S\]
The \emph{non-wandering set} $\on{NW}(\Phi)$ of a flow is the set of points that are not wandering. Note that every periodic orbit lies in the non-wandering set. Moreover, the non-wandering points and set depend only on the underlying singular line field $L$.

\begin{definition}[Morse-Smale] A singular line field $L$ on a smooth manifold $\Sigma$ is \emph{Morse-Smale} if
\begin{itemize}
    \item[(a)] the non-wandering set consists of finitely many non-degenerate closed hyperbolic orbits.
    \item[(b)] the stable and unstable manifolds of any pair of periodic orbits $\gamma$ and $\eta$ are transverse.
\end{itemize}
Note that periodic orbits include singularities in this definition. \end{definition}

\begin{example}[Gradient-Like] A vector field (or singilar line field) that is gradient-like with respect to a Morse function and that satisfies the Smale transversality condition is Morse-Smale. These are precisley the Morse-Smale vector fields that have no closed orbits (only singularities).
\end{example}

\begin{thm}[Morse-Smale Criterion] \label{thm:MS_criterion} A contact Hamiltonian manifold $(\Sigma,\eta)$ with Morse-Smale characteristic foliation is convex. \end{thm}

\begin{cor}[Gradient-like Smale] \label{cor:GL_criterion} A contact Hamiltonian manifold $(\Sigma,\eta)$ whose characteristic foliation is gradient-like, and that satisfies the Smale transversality condition, is convex.
\end{cor}

\noindent Theorem \ref{thm:MS_criterion} is due to Giroux \cite{g1991} in dimension three and to Breen \cite{b2021} in general. Corollary \ref{cor:GL_criterion} is due to Honda-Huang in higher dimensions \cite{hh2019}. We now sketch the proof of Theorem \ref{thm:MS_criterion}. 

\vspace{3pt}

\begin{proof} (Sketch) The basic problem is to construct the splitting of $\Sigma$ into a positive half $\Sigma_+$ and negative half $\Sigma_-$ with the dividing set $\Gamma$ in the middle. Under any splitting of this kind, the non-wandering set of the characteristic foliation
\[\on{NW}(\Sigma_\eta) \subset \Sigma\]
must itself divide into two halves contained in $\Sigma_+$ and $\Sigma_-$. Indeed, all of the points in $\Gamma$ are wandering since the flowline through such a point must flow into $\Sigma_-$ and out of $\Sigma_+$, never to return to $\Gamma$. Our strategy is to first identify the correct splitting
\begin{equation}
\label{eq:NW_splitting} \NW(\Sigma_\eta) = \NW_+(\Sigma_\eta) \sqcup \NW_-(\Sigma_\eta) \qquad\text{where}\qquad \NW_\pm(\Sigma_\eta) = \Sigma_\pm \cap \NW(\Sigma_\eta) \end{equation}
and then to reconstruct $\Sigma_+$ and $\Sigma_-$ from this splitting of the non-wandering set. To do this, note that the characteristic foliation $\Sigma_\eta$ is Morse-Smale. Thus the set $\NW(\Sigma_\eta)$ is precisely the set of hyperbolic periodic points (and singularities) of the characteristic foliation. Lemmas \ref{lem:breen_divergence} and \ref{lem:hh_divergence} then completely determine which of these closed orbits and singularities are positive Liouville and which are negative Liouville. This determines the splitting (\ref{eq:NW_splitting}) completely.

\vspace{3pt}

This reduces the proof to the problem of reconstructing $\Sigma_+$ and $\Sigma_-$ from the non-wandering components $\NW_+(\Sigma_\eta)$ and $\NW_-(\Sigma_\eta)$. In fact, this is essentially the problem of reconstructing a handlebody decomposition from the critical points of a graditent-like vector field. 

\begin{figure}[h]
    \centering
    \includegraphics[width=.95\linewidth]{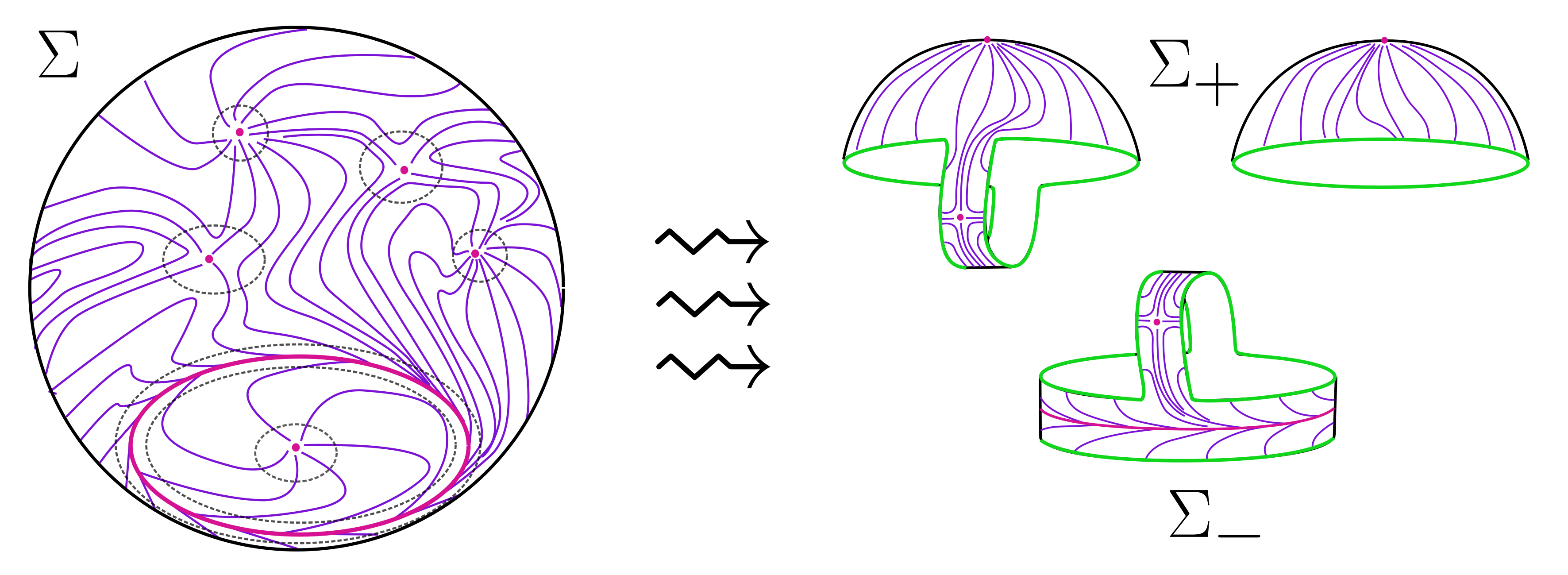}
    \caption{The splitting of a contact Hamiltonian 2-manifold with Morse-Smale characteristic foliation, reconstructed from critical points. }
    \label{fig:contact_Hamiltonian_examples}
\end{figure}

Precisely, consider the simplified setting where $\Sigma_\eta$ is gradient-like and there are no critical points of middle index $n$. In this case, $\NW_+(\Sigma_\eta)$ and $\NW_-(\Sigma_\eta)$ is simply the set of all of the critical points of index below $n$ and above $n$, respectively. Since $\Sigma_\eta$ satisfies the Smale property, there is a corresponding handle decomposition
\[
\Sigma = \Sigma_+ \cup \Sigma_-
\]
into the handlebody $\Sigma_+$ of index $n-1$ and lower handles and the handlebody $\Sigma_-$ of index $n+1$ and higher handles. The key point is that desired contact Hamiltonian forms $\lambda_+$ and $\lambda_-$ restricting to positive Liouville forms on $\Sigma_+$ and $\Sigma_-$ can now be constructed inductively, handle-by-handle. In particular, there is sufficient flexibility in the choice of positive (or negative) Liouville form near each critical point, so that a Liouville form on the union of the $k$-handles can be extended over a neighborhood of the $(k+1)$-critical points. This argument is carried out carefully in \cite[\S 2]{hh2019}. 

\vspace{3pt}

In the general case where the characteristic foliation $\Sigma_\eta$ has closed orbits, a similar reconstruction procedure can be carried out. While singularities still correspond to ordinary handles, closed hyperbolic orbits correspond to certain round handles that resemble an ordinary handle times a circle. These round handles correspond to the tubular neighborhood of the stable manifold of the given orbit. This proof is carried out in detail by Breen \cite{b2021}.\end{proof}

\subsection{Genericity In Dimension Two} \label{subsec:convex_genericity} The most important applications of the Morse-Smale criterion (Theorem \ref{thm:MS_criterion}) are various results that address the abundance of convexity among hypersurfaces in contact manifolds, due to Giroux \cite{g1991} and Honda-Huang \cite{hh2019}. A priori, it could be the case that a given contact manifold has few non-trivial convex hypersurfaces. To the contrary, the results of Giroux and Honda-Huang state that convex hypersurfaces are abundant. 

\vspace{3pt}

Let us start by explaining the smooth genericity result due to Giroux \cite{g1991}. Here is the precise statement of this result.

\begin{thm}[Giroux Genericity] \label{thm:giroux_genericity} Convex surfaces are $C^\infty$-generic in the space of smooth embedded surfaces in any contact 3-manifold $(Y,\xi)$.
\end{thm}

\noindent This result is follows more-or-less immediately from Theorem \ref{thm:MS_criterion} along with the analogue (for characteristic foliations) of the following famous result of Peixoto.

\begin{thm}[Morse-Smale Genericity In Surfaces] \label{thm:MS_genericity_surfaces} Morse-Smale flows are $C^\infty$-generic in the space of all smooth flows on any closed surface $\Sigma$.
\end{thm}

Peixoto's result was a landmark result in low-dimensional dynamics, and it is especially remarkable due to its utter failure in higher dimensions. The proof is fairly straightforward using well-known facts in smooth dynamics, so we sketch it here.

\begin{proof} (Sketch) The Poincar\'{e}-Bendixson Theorem states that the non-wandering set of a surface flow is a union of fixed points (singularities), closed orbits and graphs consisting of saddle point singularities connected by heteroclinic trajectories. 

\begin{figure}[h]
    \centering
    \includegraphics[width=.8\linewidth]{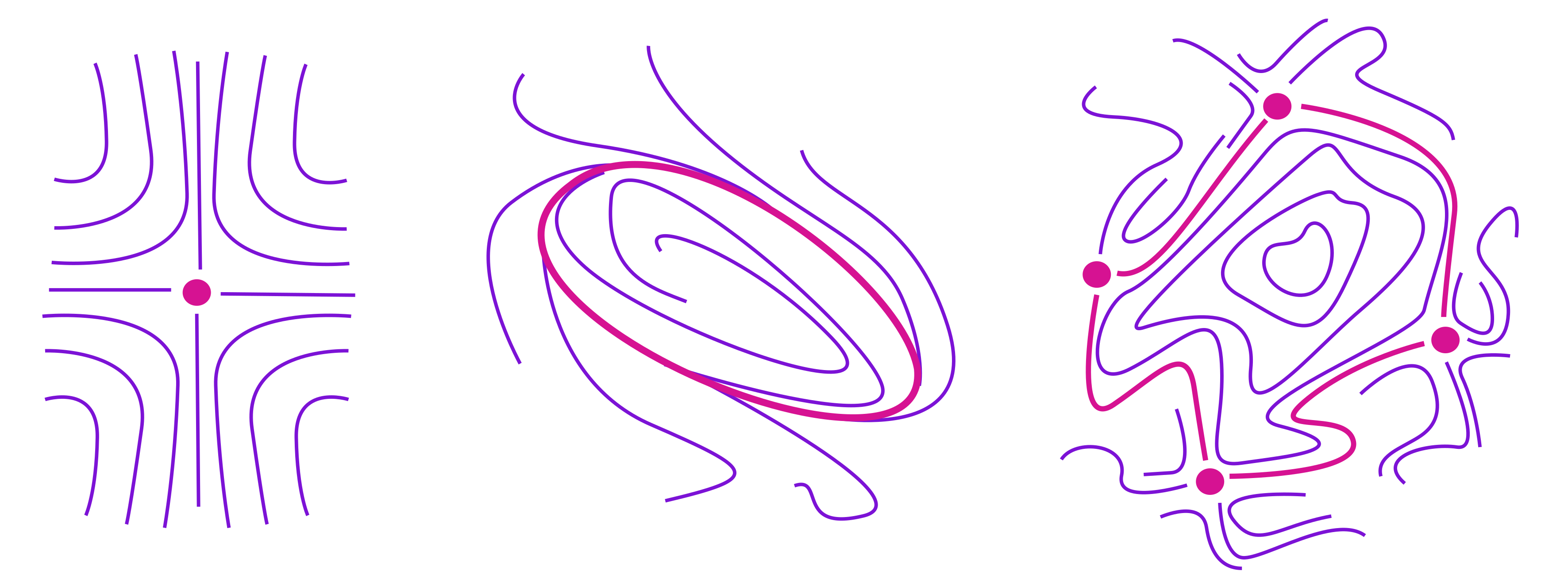}
    \caption{The three types of non-wandering components. }
    \label{fig:non_wandering_surface}
\end{figure}
\noindent We can apply the Kupka-Smale theorem \cite{kupka1963contributiona,smale1967differentiable} to $C^\infty$-approximate the flow by a Kupka-Smale flow where all of the singularities and periodic orbits are non-degenerate, and where the unstable and stable manifolds of the singularities intersect transversely. This prevents any heteroclinic connections between saddle points, leaving only singularities and closed orbits. The finiteness of the critical points and orbits follows from the fact that the non-wandering set is compact, and each singularity and closed orbit is isolated.
\end{proof}

\noindent The analogue of Theorem \ref{thm:MS_genericity_surfaces} in the setting of characteristic surfaces states that surfaces in contact 3-manifolds with Morse-Smale characteristic foliation are $C^\infty$-generic. Theorem\ref{thm:giroux_genericity} then follows from Theorem \ref{thm:MS_criterion}. 

\subsection{Density In Higher Dimensions} \label{subsec:convex_density_higher_d} In higher dimensions, a density (or genericity) result for convex hypersurfaces in the spirit of Theorem \ref{thm:giroux_genericity} remained completely open for several decades. It was widely suspected that the genericity result failed in dimensions higher than two, although proposed counter-examples of Mori \cite{mori2009reeb,mori2011} turned out to be incorrect due to the work of Breen \cite{b2021}.. This situation was finally addressed in the 2020s by groundbreaking work of Honda-Huang \cite{hh2019} who proved the following result.

\begin{thm}[Honda-Huang Density] \label{thm:honda_huang} Convex hypersurfaces are $C^0$-dense in the space of smooth embedded surfaces in any contact manifold $(Y,\xi)$ of any dimension.
\end{thm}

Like Theorem \ref{thm:giroux_genericity}, the proof is an application of the Morse-Smale criterion (Theorem \ref{thm:MS_criterion}) and the following density result for Morse-Smale (and in fact, gradient-like) characteristic foliations.

\begin{thm}[Gradient-like Density] \label{thm:gradient_like_density} Fix a contact Hamiltonian structure $\eta$ on a closed manifold $\Sigma$. Then there is an arbitrarily $C^0$-small $C^\infty$-isotopy $\eta_s$ ambient in $C\Sigma$ such that
\[
\eta_0  = \eta \qquad\text{and}\qquad \eta_1 \text{ has gradient-like characteristic foliation}
\]
Moreover, the trajectories of the characteristic foliation can be made arbitrarily short with respect to any chosen Riemannian metric. \end{thm}

\noindent Theorem \ref{thm:gradient_like_density} has an extremely complicated proof with subtleties that are well beyond the scope of this article. A simplified proof of the result was given by Eliashberg-Pancholi \cite{ep2022}. Let us sketch the latter proof of Theorem \ref{thm:gradient_like_density} in \cite{ep2022} at a very rough level. 

\vspace{3pt}

 One starts the proof by picking a transversal contact sub-manifold $\Gamma \subset \Sigma$ given by a disjoint union of many small contact disks of codimension one. This transversal must be a \emph{blocking collection},  in the sense that every trajectory of the characteristic foliation hits the interior of $\Gamma$ after some small bounded amount of length. This map can be extended to an embedding
\[
U = [0,1]_s \times \Gamma \to \Sigma
\]
of a small segment of symplectization of $\Gamma$ into $\Sigma$. We then modify the contact Hamiltonian structure and the characteristic foliation on the image of $U$ by replacing each of the components with a certain plug that is ambiently homotopic to the original contact Hamiltonian structure. 

\vspace{3pt}

Let us briefly describe the required properties of the plug and give a picture. Let $P$ denote the product $[0,1] \times \mathbb{D}^m$ and $L_{\on{std}}$ denote the line field spanned by $\partial_s$. We decompose the boundary as
\[
\partial_{\on{i}} P = 0 \times \mathbb{D}^m \qquad \partial_{\on{o}}P = 1 \times \mathbb{D}^m \qquad \partial_{\on{t}}P = [0,1]_s \times \partial\mathbb{D}^m
\]
These are respectively the \emph{inward boundary}, \emph{outward boundary} and \emph{tangent boundary}, named for the fact that the vector field $L_{\on{std}}$ points into, out of, and tangent to $P$ in these parts of the boundary.

\begin{definition}[Plug] \label{def:plug} A \emph{gradient-like plug} of trapping size $\epsilon$ is an oriented singular line field
\[
L \qquad\text{on the standard tube}\qquad P = [0,1]_s \times \D^m
\]
with the following properties.
\begin{itemize}
    \item[(a)] (Boundary Standard) The singular line field $L$ is the span of $L_{\on{std}} = \on{span}(\partial_s)$ near $\partial P$.
    \vspace{2pt}
    \item[(b)] (Gradient-Like) The singular line field is Morse-Smale and gradient-like.
    \vspace{2pt}
    \item[(c)] (Limit Sets) Any trajectory $\gamma$ of $L$ starting at a point $p$ in $\partial_{\on{i}}P$ such that $\on{dist}(p,\partial_{\on{t}}P) > \epsilon$ must limit to a critical point in the forward direction.
    \vspace{2pt}
    \item[(d)] (Monodromy) Any trajectory of $L$ starting at $p = 0 \times x\in \partial_{\on{i}}P$ either converges to a critical point of $L$ or ends at a point $1 \times y \in \partial_{\on{o}}P$ with $\on{dist}(x,y) < \epsilon$ where $\on{dist}$ denotes the distance in the disk with respect to the standard metric.
    \end{itemize}\end{definition}

    \begin{figure}[h]
    \centering
    \includegraphics[width=.9\linewidth]{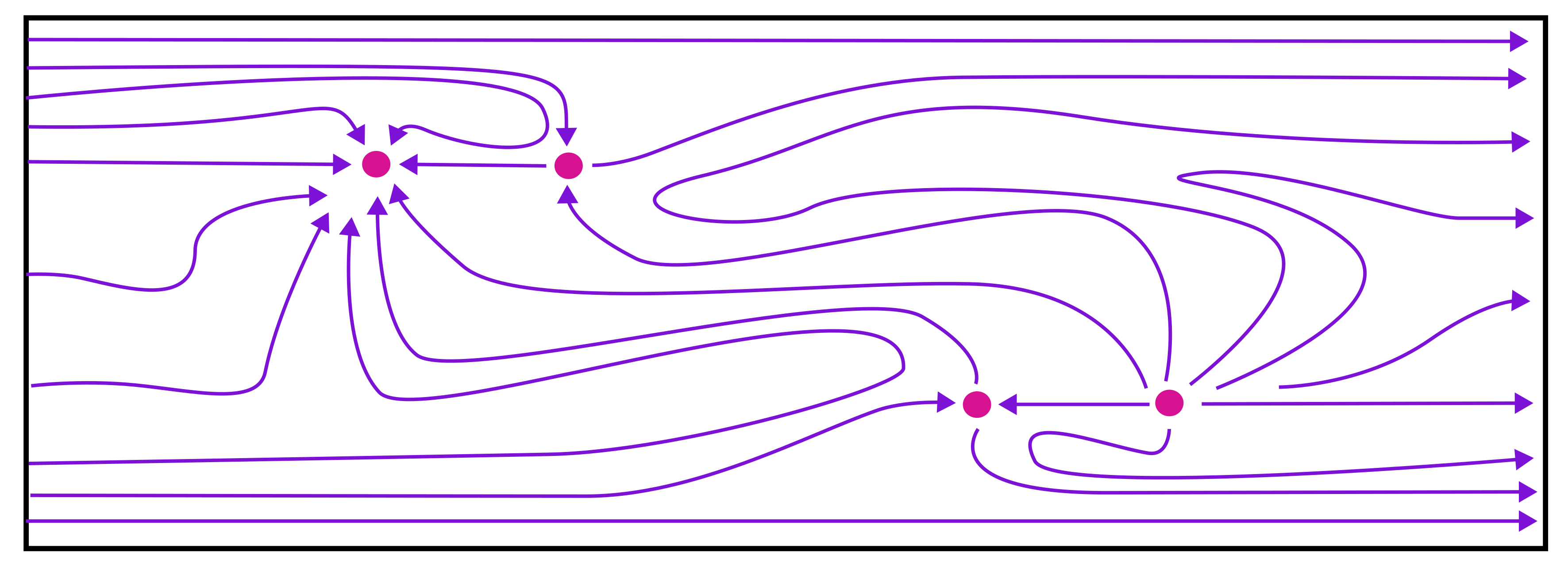}
    \caption{A 2-dimensional plug for small trapping size. }
    \label{fig:plug}
\end{figure}

\begin{remark} \cite[Rmk 2.5]{ep2022} A gradient-like plug of trapping size $\epsilon$ is also a gradient-plug of trapping size $2\epsilon$ for when the orientation of $L_{\on{std}}$ is reversed. In particular, trajectories that end on $\partial_{\on{o}}P$ far away from the boundary must originate at critical points.
\end{remark}

\begin{definition}[Contact Hamiltonian Plug] A \emph{contact Hamiltonian plug} $\eta$ of trapping size $\epsilon$ and height $\delta$ is a contact Hamiltonian structure $\eta$ on $P = [0,1]_s \times \D^{2n-1}$ such that
\begin{itemize}
    \item[(a)] The characteristic foliation $\Sigma_\eta$ is a gradient-like plug of trapping size $\epsilon$.
    \item[(b)] The contact Hamiltonian strucure $\eta$ is ambiently homotopic to the standard structure $\eta_{\on{std}} = \on{span}(\partial_s) \oplus \xi_{\on{std}}$ via an isotopy in the contactization $CP$ of $C^0$-size less than $\delta$. 
\end{itemize}\end{definition}

Given a blocking collection $\Gamma$ thickened to an embedding $[0,1] \times \Gamma \to \Sigma$, one may replace each component with a gradient-like plug to get a \emph{plugged} singular line field. One may show that, if the trapping size is small enough, then every trajectory of the plugged line field must begin and end on a non-degenerate singularity in one of the plugs. After a further $C^\infty$-small perturbation, the plugged flow becomes gradient-like and Morse-Smale. This discussion reduces Theorem \ref{thm:gradient_like_density} to the following existence result for plugs (cf. \cite[Prop 4.1]{ep2022}).

\begin{thm}[Plug Existence] \label{thm:CH_plug_existence} For every $\delta$ and $\epsilon$, there is a contact Hamiltonian plug $(P,\eta)$ of trapping size $\epsilon$ and height $\delta$.
\end{thm}

\noindent Theorem \ref{thm:CH_plug_existence} is where the majority of the difficulty in Theorem \ref{thm:gradient_like_density} lies. The proof given by Eliashberg-Pancholi \cite{ep2022} follows several stages. We conclude this part by discussing these steps.

\vspace{3pt}

First, one proves that one can construct a contact Hamiltonian plug of height $\delta$ from a contact Hamiltonian plug of any height (cf. \cite[\S 5]{ep2022}). This is accomplished by constructing the small height plug by layering many small copies of the large height plug and exploiting an advantageous scaling law for the height under scaling of the region $[0,1] \times \D^{2n-1}$. This allows us to ignore the height in the remainder of the proof.

\vspace{3pt}

Second, one constructs a contact Hamiltonian plug of trapping size $\epsilon$ for any $\epsilon$ in dimension two. This can be done extremely explicitly, and a picture of such a plug is given in Figure \ref{fig:plug}. Denote the plug of dimension two by
\[
(P^2,\eta_\epsilon)
\]
Using the 2-dimensional plug, one can construct certain higher dimensional \emph{quasi-plugs}. A quasi-plug is similar to a contact Hamiltonian plug where $P$ is replaced by
\[
Q = [0,1] \times Y
\]
where $Y$ is contact manifold with boundary. Here one may view this as the cylinder $[0,1] \times CW$ where $CW$ is the contactization of the Weinstein domain $W$. The definition of a quasi-plug of trapping size $\epsilon$ is the same as the definition of a plug, except for the monodromy property (Definition \ref{def:plug}(d)) which is replaced the following modified monodromy
\begin{itemize}
    \item[(d')] (Monodromy') Let $\gamma$ trajectory of the characteristic foliation of $Q$ starting at $0 \times x \in \partial_{\on{i}}Q$ and ending at $1 \times y\in \partial_{\on{o}}Q$. Then there is a characteristic flowline $\zeta$ of the characteristic foliation of $\partial Y$, as a hypersurface in $Y$, such that $\on{dist}(x,\zeta) < \epsilon$ and $\on{dist}(y,\zeta) < \epsilon$.
\end{itemize}
Note that if the flowlines of the characteristic foliation of $\partial Y$ are all shorter than $\epsilon$, then an quasi-plug of trapping size $\epsilon$ is a plug of trapping size $\epsilon$ (with the disk replaced by a different contact manifold). 

\vspace{3pt}

Finally, one constructs a plug on $P$. This step is carried out by perturbing the disk $\mathbb{D}^{2n-1}$ to a new domain $Y \subset \mathbb{D}^{2n-1}$ such that $Y$ has very short Morse-Smale characteristic foliation, then using the quasi-plug construction from the previous step to construct a plug on $[0,1] \times Y$, which is then also a plug on $[0,1] \times \D^{2n-1} = P$. An interesting detail of this step is that it requires induction on dimension: one must have already proven the approximation result Theorem \ref{thm:gradient_like_density} in order to find the approximation $Y$.

\noindent 

\subsection{Non-Density In Higher Dimensions} A longstanding question following the work of Giroux was the extent to which the smooth genericity of convexity (Theorem \ref{thm:giroux_genericity}) extends to higher dimensions. The work of Honda-Huang \cite{hh2019} did not resolve this problem \cite[Rmk 1.2.4]{hh2019} and moreover their methods were not applicable to its resolution, due the evident failure of the density of gradient-like characteristic foliations in even the $C^1$-topology. This situation was finally rectified by recent work of the author \cite{jc2024} where the following result was proven.

\begin{thm}[Robust Non-Convexity] \label{thm:robust_non_convexity} There are closed oriented hypersurfaces in any contact manifold of dimension five or greater that cannot be approximated by convex hypersurfaces in the $C^2$-topology. 
\end{thm}

\noindent Based on this result and inspired by terminology from the smooth dynamics literature, we introduced the following terminology in \cite{jc2024}.

\begin{definition}[Robust Non-Convexity] A contact Hamiltonian manifold $(\Sigma,\eta)$ is \emph{robustly non-convex} if $\eta$ cannot be approximated in the $C^2$-topology by a convex contact Hamiltonian structure.
    
\end{definition}

\noindent The mechanism for robust non-convexity is dynamical, and it hints at a fundamental relationship with generic dichotomies in differentiable dynamics as pioneered by Bonatti-Diaz \cite{bd2008,bdp2003,bonattidiaz1995} and others. Let us conclude this section by sketching the proof of Theorem \ref{thm:robust_non_convexity}.

\vspace{3pt}

The starting point for the proof of Theorem \ref{thm:robust_non_convexity} is the simple observation that the characteristic foliation of a convex hypersurface (or contact Hamiltonian manifold) cannot be transitive.

\begin{definition}[Transitive] A smooth flow $\Phi$ on a manifold $M$ is \emph{transitive} if it has a dense orbit.
\end{definition}

\begin{lemma}[Non-Transitivity] Let $(\Sigma,\eta)$ be a convex contact Hamiltonian manifold. Then the characteristic foliation $\Sigma_\eta$ (or equivalently, any characteristic flow) is not transitive.
\end{lemma}

\begin{proof} If $(\Sigma,\eta)$ is convex, then admits a dividing set $\Gamma$ bounding a sub-domain  $\Sigma_- \subset \Sigma$ such that $\Sigma_\eta$ is transverse to $\Gamma$, pointing into $\Sigma_-$. Therefore, any trajectory can only visit a small neighborhood of $\Gamma$ once, which makes density impossible. 
\end{proof}

Theorem \ref{thm:robust_non_convexity} is thus reduced to the problem of constructing a contact Hamiltonian manifold $(\Sigma,\eta)$ with robustly transitive characteristic foliation. To understand this problem, consider the analogous problem in the setting of ordinary differentiable dynamics. In that setting, the simplest example of robustly transitive flows are transitive Anosov flows, which are structurally stable. Unfortunately, Theorem \ref{thm:no_Anosovs} shows that we cannot hope to find such examples in the setting of characteristic foliations. 

\vspace{3pt}

Luckily, the problem of finding non-hyperbolic and robustly transitive dynamical systems was addressed in the smooth setting by seminal work of Bonatti-Diaz \cite{bonattidiaz1995}. The key step to their construction was the construction of an object called a \emph{blender} \cite{whatisblender}, which is the primary mechanism for such phenomena. Roughly speaking, a blender is a type of structure consisting of two hyperbolic periodic orbits of complementary index, whose stable and unstable manifolds intersect robustly along a larger-than-expected set. In particular, this produces a robust cycle of heteroclinics connecting the two points. Blenders were introduced in \cite{bonattidiaz1995} and subsequently became a fundamental tool in differentiable dynamics \cite{bonatti2004dynamics,bd2008,whatisblender}. 

\vspace{3pt}

The construction of a blender given in e.g. \cite{bonattidiaz1995} is delicate and several obstacles appear in the adaptation of this construction to the contact case. Regardless, this can be done and this results in the following theorem.

\begin{figure}[h]
    \centering
    \includegraphics[width=.9\linewidth]{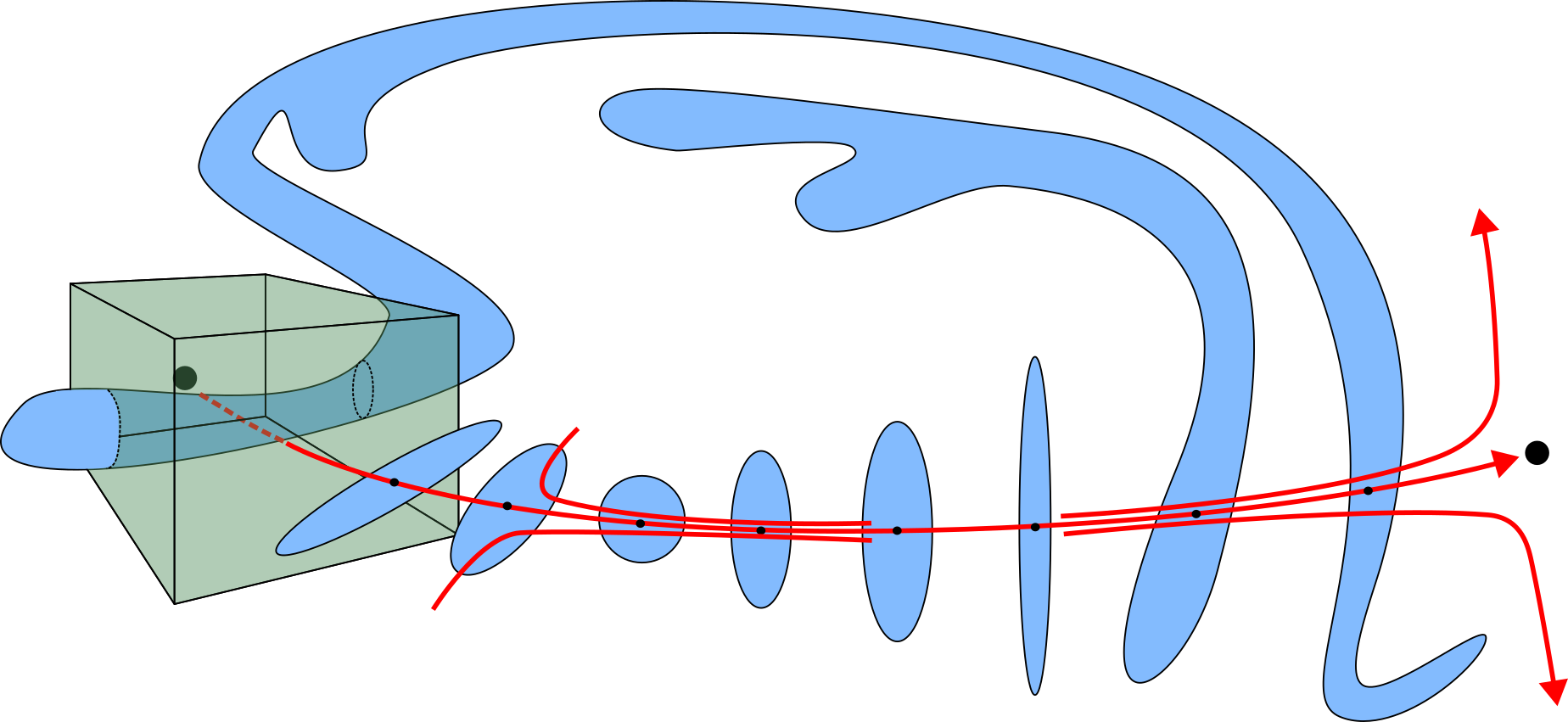}
    \caption{A cartoon of a blender. Here the two points in black are the interacting periodic points and the red curve is one of the heteroclinics.}
    \label{fig:blender}
\end{figure}

\begin{thm} \cite{jc2024} \label{thm:robust_contactomorphism} Let $S\Lambda$ be the unit cosphere bundle of a hyperbolic manifold and let $\Phi:S\Lambda \to S\Lambda$ be the time-$T$ map of the (hyperbolic) geodesic flow where $T$ is the period of a closed orbit. Then there is a $C^\infty$-small, $C^1$-robustly transitive contact perturbation $\Psi$ containing a blender $B$.
\end{thm}

Using the suspension construction for contactomorphisms (Example \ref{ex:suspensions}), this can be used to produce a contact Hamiltonian structure $\eta$ with $C^2$-robustly transitive characteristic foliation on the product $\Sigma = \R/\Z \times S\Lambda$. For a specific family of hyperbolic manifolds $\Lambda$, one can find Legendrian embeddings $\Lambda \to \R^{2n+1}$ into standard contact $(2n+1)$-space using Murphy's h-principle \cite{m2012}. This can then be used to construct an embedding $(\Sigma,\eta) \to \R^{2n+1}$, yielding Theorem \ref{thm:robust_non_convexity}. 

\begin{remark}[Topologies] The transition from the $C^2$-topology in Theorem \ref{thm:robust_non_convexity} to the $C^1$-topology in Theorem \ref{thm:robust_contactomorphism} comes from the fact that the map from contact Hamiltonian forms to characteristic vector fields induced by (\ref{eq:characteristic_vectorfield}) to is continuous in the $C^2$-topology on the domain and the $C^1$-topology on the target. 
\end{remark}

\noindent The proof of Theorem \ref{thm:robust_non_convexity} suggests a potential connection between robust non-hyperbolicity and convexity. We will discuss some conjectures and questions in this direction in the next section.

\section{Problems, Questions And Conjectures} \label{sec:questions_and_conjectures}

We conclude with a discussion of various problems, questions and conjectures about the topology of contact Hamiltonian manifolds and the dynamics of their characteristic foliations.

\subsection{Existence And Tightness} \label{subsec:tightness} It is natural to ask when a specific even-dimensional manifold admits a contact Hamiltonian structure, or more restrictively a convex structure. The analogous problem for symplectic and contact structures has been a major topic of research since the emergence of symplectic topology (cf. \cite[Problem 1]{ms2017}). 

\vspace{3pt}

The formal existence problem for contact Hamiltonian structures (cf. \cite{cieliebak2024introduction}) is equivalent to the formal existence problem for contact structures on the contactization. More precisely, recall the following definition.  

\begin{definition}[Almost Contact] A (cooriented) \emph{almost contact structure} $E$ on an oriented $(2n+1)$-manifold $Y$ consists of a cooriented hyperplane field
\[E \subset TY \text{ with a fiberwise symplectic structure }\omega\]
\end{definition}

\noindent The existence problem for contact structures was largely addressed by the fundamental work of Borman-Eliashberg-Murphy \cite{borman2015existence}. There the authors introduced the notion of an \emph{overtwisted} contact structure on a contact manifold in any dimension, and proved the following result.

\begin{thm}[Contact H-Principle] Any almost contact structure $E$ on a closed manifold $Y$ is isotopic (through almost contact structures) to an overtwisted contact structure $\xi$ on $Y$. \end{thm}

\noindent This implies a corresponding existence result for convex contact Hamiltonian structures.

\begin{prop}[Existence] \label{prop:exitence_of_CHS} Let $\Sigma$ be a closed oriented $2n$-manifold and let $E$ be an almost contact structure on $[-1,1] \times \Sigma$. Then there is an convex contact Hamiltonian structure $\eta$ on $\Sigma$ such that the induced contact structure $\xi$ on the contactization $C\Sigma = [-1,1] \times \Sigma$ is homotopic to $E$.
\end{prop}

\begin{proof} By applying the Borman-Eliashberg-Murphy $h$-principle  for contact structures \cite{borman2015existence} we can find an overtwisted contact structure $\xi$ on $[-1,1] \times \Sigma$ homotopic to $E$. By applying Theorem \ref{thm:honda_huang}, we can then perturb the embedding
\[
\Sigma = 0 \times \Sigma \to [-1,1] \times \Sigma \qquad\text{to a $C^0$-close embedding}\qquad \iota:\Sigma \to [-1,1] \times \Sigma
\]
such that $\eta = \iota^*\xi$ is convex. The contactization $C\Sigma = [-1,1] \times \Sigma$ with respect to $\eta$ then embeds into $[-1,1] \times \Sigma$ by Lemma \ref{lem:std_nbhd} by an embedding that is homotopic to the identity. It follows that the contact form on the contactization is homotopic to $E$. 
\end{proof}

Contact Hamiltonian manifolds possess a tight-overtwisted dichotomy that is inherited from the contact setting. In particular, we adopt the following terminology.

\begin{definition}[Overtwisted] A contact Hamiltonian structure $\eta$ on a manifold $\Sigma$ is \emph{overtwisted} if any neighborhood of $\Sigma$ in its contactization is overtwisted. 
\end{definition}

\noindent It seems likely that the contact Hamiltonian structure constructed in Proposition \ref{prop:exitence_of_CHS}, via the Borman-Eliashberg-Murphy existence result, is overtwisted. However, this is not immediate and thus we pose the following question.

\begin{question} Is the contact Hamiltonian structure $\eta$ from Proposition \ref{prop:exitence_of_CHS} always overtwisted?
\end{question}

\noindent More generally, it is natural to ask if every contact Hamiltonian manifold admits a tight contact Hamiltonian manifold. 

\begin{question}[Tight Structures] \label{qu:tight_structures} Let $\Sigma$ be a closed 4-manifold (or more generally, $2n$-manifold) such that $[-1,1] \times \Sigma$ is almost contact. Does $\Sigma$ admit a tight contact Hamiltonian structure? 
\end{question}

\begin{remark} It is well known by work of Etnyre-Honda \cite{e2001} that there are closed 3-manifolds admitting no tight contact structures. Question \ref{qu:tight_structures} asks for a similar example in dimension four.
\end{remark}

\begin{remark} By applying Theorem \ref{thm:honda_huang}, it is simple to show that any $2n$-manifold admitting a tight contact Hamiltonian structure also has a tight and convex one.
\end{remark}

\subsection{Doubles} There is a rich source of contact Hamiltonian manifolds that are both tight and convex, by taking certain kinds of doubles of Lioiville domains. 

\begin{example}[Doubles] Let $(W,\lambda)$ denote the Weinstein (or more generally, Liouville) domain and $\Phi:\partial W \to \partial W$ be a contactomorphism of the boundary. The \emph{double} is given by
\[
\Sigma(W,\Phi) = W \underset{\Phi}{\cup} \bar{W}
\]
The contact Hamiltonian structures $\eta = \on{ker}(\lambda)$ on $W$ and $\bar{W}$ glue together to give a convex contact Hamiltonian structure on $\Sigma(W,\Phi)$. Moreover, this structure only depends only on the contact isotopy classes of $\Phi$ up to isotopy through convex contact Hamiltonian structures. \end{example}

\begin{lemma} The double $(\Sigma(W,\Phi),\eta)$ is tight if $\Phi$ extends to an exact symplectomorphism $W \to W$.
\end{lemma}

\begin{proof} This follows from the fact that the contact homology of a convex sutured neighborhood of $\Sigma(W,\Phi)$ is non-vanishing. Indeed, let $\Gamma = \partial W$ and consider the augmentations
\[
\epsilon_+:A(\Gamma) \to \Q \qquad\text{and}\qquad \epsilon_-:A(\Gamma) \to \Q
\]
induced by the fillings $(W,\on{Id})$ and $(W,\Phi)$. Here it is key to specify the precise identifications $\on{Id}$ and $\Phi$ of $\Gamma$ with the contact boundaries of the fillings in question. By Avdek \cite[Thm 1.1.1]{a2023}, the contact homology of a neighborhood of $\Sigma(W,\Phi)$ is non-vanishing if and only if $\epsilon_+$ is dga-homotopic to $\epsilon_-$. In the case where $\Phi$ extends to a symplectomorphism of $W$, one can deduce that in fact $\epsilon_+ = \epsilon_-$ for appropriate choice of data (e.g. almost complex structures, Kuranishi data, etc). This yields the desired result.\end{proof}

\noindent In contrast to the above result, it is also possible for a double $\Sigma(W,\Phi)$ to be overtwisted. Indeed, this occurs in the case of an overtwisted bypass as discussed by Honda-Huang \cite{hh2018}. 

\vspace{3pt}

It is interesting to ask if the overtwistedness of a double can be used to detect exotic symplectic phenomena. In particular, we pose the following question.

\begin{question}[Overtwisted Doubles] \label{qu:overtwisted_doubles} Does there exist a of a Liouville $2n$-manifold $W$ such that
\[
\Phi \text{ extends as a diffeomorphism over $W$}\qquad\text{and}\qquad \Sigma(W,\Phi) \text{ is overtwisted}
\]
\end{question}

\noindent Any contactomorphism as in Question \ref{qu:overtwisted_doubles} would represent a type of relative exotic behavior distinguishing the symplectic and smooth categories.

\subsection{Convexity Criteria} The Morse-Smale criterion for convexity is an essential ingredient in the proofs in Section \ref{subsec:convex_genericity} and \ref{subsec:convex_density_higher_d}. This motivates the following problem.

\begin{problem}[Convexity Criteria] Give (necessary and sufficient, or generic) criteria, via the dynamics of the characteristic foliation, for a contact Hamiltonian manifold $(\Sigma,\eta)$ to be convex.
\end{problem}

\begin{remark} Here a criterion is \emph{generic} if the criterion implies convexity for a $C^2$-comeager subset of all contact Hamiltonian structures.
\end{remark}

\noindent The Morse-Smale criterion in Theorem \ref{thm:MS_criterion} for convexity is an example of a sufficient convexity criterion, which is generically necessary and sufficient in dimension two. Similarly, the gradient-like condition is generically sufficient (since the Smale property is $C^\infty$-generic). We expect that more general criteria in higher dimensions will many interesting applications.

\vspace{3pt}

In forthcoming work, the author will present candidate criteria for convexity that generalize the Morse-Smale criterion via Conley theory \cite{conley1978isolated}. Let us sketch some of these results. Recall that the \emph{chain recurrent set} of a vector field $Z$ on a closed manifold $\Sigma$ is a closed invariant set
\[
\on{CR}(Z) \subset \Sigma
\]
consisting of \emph{chain recurrent} points. Chain recurrence is the weakest possible type of recurrence, generalizing periodicity and non-wandering. The chain recurrent set decomposes into connected \emph{chain components} $\Lambda$. The chain components possess a natural ordering $\limitsto$ where $\Lambda \limitsto \Xi$ if the chain component $\Lambda$ is there is a generalized heteroclinic trajectory from $\Lambda$ to $\Xi$. 

\begin{definition} A chain component $\Lambda$ of a vector field $Z$ is \emph{positive} if there exists a volume form $\mu$ with positive divergence along $\Lambda$. We define \emph{negative} chain components similarly.
\end{definition}

\begin{remark} There is actually a formulation of positive and negative chain components coming from ergodic theory. Precisely, there is a well-defined element
\[
\on{div}(Z,\mu)(Z,\on{Leb}) \in M(\Sigma,Z)^\vee 
\]
in the dual space of $Z$-invariant signed measures $M(\Sigma,Z)$. This is defined by integration against the divergence $\on{div}(Z,\mu)(Z,\mu)$ of any volume form with respect to the vector field, where
\[
\mathcal{L}_Z\mu = \on{div}(Z,\mu) \cdot \mu
\]
One may check that a chain component $\Lambda$ is positive if and only
\[
\langle \on{div}(Z,\on{Leb}), \nu\rangle > 0 \qquad\text{for every invariant measure $\nu$ with $\on{supp}(\nu) \subset \Lambda$}
\]\end{remark}

\begin{definition} \label{def:chain_convexity} A vector-field $Z$ is \emph{chain convex} if the following conditions hold.
\begin{itemize}
\item[(a)] Every chain component $\Lambda$ is positive or negative. 
\vspace{2pt}
\item[(b)] There does not exist a negative chain component $\Lambda$ and a positive chain component $\Xi$ such that $\Lambda \limitsto \Xi$. We call this a \emph{retrograde connection}.
\end{itemize}
\end{definition}

\noindent The following result result will establish chain convexity as a necessary and sufficient condition for convexity of hypersurfaces (or contact Hamiltonian structures). A proof uses the existence of certain nice Lyapunov functions along with an inductive argument similar to the Morse-Smale case. The details will appear elsewhere.

\begin{thm} (Forthcoming) A contact Hamiltonian manifold $(\Sigma,\eta)$ is convex if and only if the characteristic foliation $\Sigma_\eta$ is chain convex.
\end{thm}

\noindent In the setting of Morse-Smale flows, the chain components consist of individual hyperbolic singularities and orbits, and the the chain relation is such that $\Lambda \limitsto \Xi$ if and only if there is a heteroclinic trajectory from $\Lambda$ to $\Xi$. The index constraints on positive and negative orbits (Lemmas \ref{lem:hh_divergence} and \ref{lem:breen_divergence}) and the Smale transversality condition then prevent retrograde connections. This is how the Morse-Smale condition is recovered from the above result.

\vspace{3pt}

In general, chain recurrent sets can be quite wild. However, the chain recurrent set of $C^1$-generic diffeomorphisms and flows are known to have more controlled behavior. In particular, there is a (partially conjectural) dichotomy stating that, $C^1$-generically, a chain component is either hyperbolic or contains a $C^1$-robust, coindex one heterodimensional cycle (cf. Bonatti-Diaz \cite{bonatti2008robust}).

\begin{definition}[Heterodimensional Cycle] A \emph{heterodimensional cycle} $C$ is a pair of heteroclinic trajectories $\Lambda \to \Xi$ and $\Lambda \to \Xi$ between a pair of hyperbolic invariant sets $\Lambda$ and $\Xi$. The \emph{coindex} is the difference in the dimension of the stable bundles of $\Lambda$ and $\Xi$.
\end{definition}

\noindent The first constructions of robust heterodimensional cycles were given via blender constructions, and such a cycle is present in the construction at the heart of Theorem \ref{thm:robust_non_convexity}. Moreover, if such a cycle contains a retrograde heteroclinic from a negative hyperbolic invariant set to a positive hyperbolic invariant set, then it provides a robust obstruction to convexity. We conjecture that, $C^2$-generically, such heteroclinic cycles are the universal convexity obstruction.

\begin{conjecture}[Cycle Criterion] A $C^2$-generic contact Hamiltonian structure $\eta$ on $\Sigma$ is either convex or has robust positive-negative heterodimensional cycles. 
\end{conjecture}

\noindent This conjecture would give a necessary and sufficient convexity condition that is $C^2$-generic: the absence of heteroclinic cycles connecting positive and negative hyperbolic invariant sets.

\subsection{Approximation} Another natural problem deals with the sharpness of the density and non-density results of Honda-Huang \cite{hh2019} and the author \cite{jc2024}. We pose the following question.

\begin{question}[$C^1$-Genericity] Are convex hypersurfaces $C^1$-generic in the space of closed hypersurfaces in any contact manifold?
\end{question}

\noindent Note that hypersurfaces with Morse-Smale characteristic foliation are not $C^1$-dense. Indeed, a $C^1$-small perturbation cannot introduce new singularities to a non-singular line field. Thus, the Morse-Smale criterion for convexity is not useful for this approximation problem. On the otherhand, the chain convexity condition in Definition \ref{def:chain_convexity} could potentially be dense. This is an interesting topic for future work.

\subsection{Liouville Flows} Finally, it is interesting to ask about the $C^1$-generic structure of Liouville flows on Liouville manifolds. This is the case of a contact Hamiltonian manifold with only positive regions. 

\begin{question}[Generic Liouville Flows] What properties of the Liouville flow on a Liouville manifold $(W,\lambda)$ are $C^1$-generic with respect to the Liouville form $\lambda$?
\end{question}

\noindent Any structural results in this direction could yield interesting constraints on the topology of Liouville manifolds. We pose the following more specific question, which is of particular interest to the author.

\begin{question}[Axiom A] Is the set of Liouville forms $\lambda$ with Axiom A Liouville flow $C^1$-dense in the space of all Liouville forms on a manifold $W$?
\end{question}
\bibliographystyle{hplain}
\bibliography{standard_bib}

\end{document}